\documentclass[12pt]{article}
\usepackage{amsmath,amsxtra,latexsym,amsthm,amssymb,amscd,amsfonts}
\usepackage[portrait, top=3.5cm, bottom=3cm, left=3.5cm, right=2cm] {geometry}
\usepackage{graphicx}
\usepackage{color}

\theoremstyle{plain}
\newtheorem{corollary}{\bf Corollary}[section]
\newtheorem{lemma}{\bf Lemma}[section]
\newtheorem{ex}{\bf Example}[section]

\newtheorem{proposition}{\bf Proposition}[section]

\newtheorem{theorem}{\bf Theorem}[section]

\def\disp{\displaystyle}

\def\Limsup{\mathop{{\rm Lim}\,{\rm sup}}}

\def\tto{\;{\lower 1pt \hbox{$\rightarrow$}}\kern -10pt
\hbox{\raise 2pt \hbox{$\rightarrow$}}\;}

\def\R{{\rm I\!R}}
\def\N{{\rm I\!N}}

\def\co{\mbox{\rm co}\,}
\def\int{\mbox{\rm int}\,}
\def\gph{\mbox{\rm gph}\,}
\def\epi{\mbox{\rm epi}\,}

\def\dom{\mbox{\rm dom}\,}

\def\cl{\mbox{\rm cl}\,}

\def\O{\Omega}

\begin{document}
\pagestyle{myheadings}

\newtheorem{Theorem}{Theorem}[section]
\newtheorem{Proposition}[Theorem]{Proposition}
\newtheorem{Remark}[Theorem]{Remark}
\newtheorem{Lemma}[Theorem]{Lemma}
\newtheorem{Corollary}[Theorem]{Corollary}
\newtheorem{Definition}[Theorem]{Definition}
\newtheorem{Example}[Theorem]{Example}
\renewcommand{\theequation}{\thesection.\arabic{equation}}
\normalsize

\setcounter{equation}{0}

\title{Differentiability Properties of a Parametric Consumer Problem}

\author{V.T.~Huong\footnote{Graduate Training Center, Institute of Mathematics, Vietnam Academy of Science and Technology, 18 Hoang
Quoc Viet, Hanoi 10307, Vietnam; email:
huong263@gmail.com; vthuong@math.ac.vn.},\ \, J.-C. Yao\footnote{Center for General Education, China Medical University, 
			Taichung 40402, Taiwan; email: yaojc@mail.cmu.edu.tw.}, \ \, and \ \, N.D. Yen\footnote{Institute
of Mathematics, Vietnam Academy of Science and Technology, 18 Hoang
Quoc Viet, Hanoi 10307, Vietnam; email: ndyen@math.ac.vn.}}

\maketitle
\date{}
\medskip
\begin{quote}
\noindent {\bf Abstract.}
We study the budget map and the indirect utility function of a parametric consumer problem in a Banach space setting by some advanced tools from set-valued and variational analysis. The Lipschitz-likeness and differentiability properties of the budget map, as well as formulas for finding subdifferentials of the infimal nuisance function, which is obtained from the indirect utility function by changing its sign, are established. Our investigation is mainly based on the paper by Mordukhovich [J. Global Optim. 28 (2004), 347--362] on coderivative analysis of variational systems and the paper of Mordukhovich, Nam, and Yen [Math. Program. 116 (2009), 369--396] on subgradients of marginal functions. Economic meanings of the obtained subdifferential estimates are explained in details.


\medskip


\end{quote}
\section{Introduction}

Various stability properties and a result on solution sensitivity of a consumer problem were obtained in our recent paper \cite{Huong_Yao_Yen_2016}. Namely, focusing on some nice features of the budget map, we established the continuity and the locally Lipschitz continuity of the indirect utility function, as well as the Lipschitz-H\"older continuity of the demand map under a minimal set of assumptions. The recent work of Penot \cite{Penot_2014} was our starting point, while an implicit function theorem of Borwein \cite{Borwein_1986} and a theorem of Yen \cite{Yen_1995}  on solution sensitivity of  parametric variational inequalities are the main tools in the proofs of \cite{Huong_Yao_Yen_2016}. Differentiability properties of the budget map and the \textit{infimal nuisance function}, which is obtained from the indirect utility function by changing its sign, will be investigated in the present paper.

For classical problems in consumption economics, the interested reader is referred to Intriligator \cite[p.~149]{Intriligator_2002} and Nicholson and Snyder \cite[p.~132]{Nicholson_Snyder_2012}. Qualitative properties of the problems of maximizing utility subject to consumer budget constraint have been studied by Takayama \cite[pp.~241--242, 253--255]{Takayama_1974}, Penot \cite{Penot_2013,Penot_2014}, Hadjisavvas and Penot \cite{Hadjisavvasa_Penot 2015}, and many other authors.  

Diewert \cite{Diewert_1974},  Crouzeix \cite{Crouzeix_1983}, Mart\'inez-Legaz and Santos \cite{Martinez-Legaz_Santos_1993}, and Penot \cite{Penot_2014}
studied the duality between the utility function and the indirect utility function.  
 
Relationships between the differentiability properties of the utility function and of the indirect utility function have been discussed by Crouzeix \cite[Sections 2 and 6]{Crouzeix_1983}, who gave sufficient conditions for the indirect utility function in finite dimensions to be differentiable. He also established \cite{Crouzeix_2008} some  relationships between the second-order derivatives of the direct and indirect utility functions. Subdifferentials of the indirect utility function in infinite-dimensional consumer problems have been computed by  Penot \cite{Penot_2013}. 
 
This paper has similar aims as those of \cite{Penot_2013}. We will adopt the general infinite-dimensional setting of the consumer problem which was used in \cite{Penot_2013}. But our approach and results are quite different from the ones of Penot \cite{Penot_2013}. Namely, by an intensive use of some theorems from Mordukhovich \cite{Mordukhovich_2004_JOGO}, we will obtain sufficient conditions for the budget map to be Lipschitz-like at a given point in its graph under weak assumptions. Formulas for computing the Fr\'echet coderivative and the limiting coderivative of the budget map can be also obtained by the results of \cite{Mordukhovich_2004_JOGO} and some advanced calculus rules from \cite{Mordukhovich_2006a}. The results of Mordukhovich \textit{et al.} \cite{Mordukhovich_Nam_Yen_2009} and the just mentioned coderivative formulas allow us to get new results on differential stability of the consumer problem where the price is subject to change. To be more precise, we establish  formulas for computing or estimating the Fr\'echet, limiting, and singular subdifferentials of the infimal nuisance function.

Subdifferential estimates for the infimal nuisance function can lead to interesting economic interpretations. Namely, we will show that if the current price moves forward a direction then, under suitable conditions, the \textit{instant rate of the change of the maximal satisfaction of the consumer} is bounded above and below by real numbers defined by subdifferentials of the infimal nuisance function.
 
The paper is organized as follows. Section 2 formulates the problem of maximizing utility subject to consumer budget constraint and recalls some tools from variational analysis that will be used in the sequel. The Lipchitz-likeness and differentiability properties of the budget map are studied in Section 3. Formulas for computing or estimating the Fr\'echet, limiting, and singular subdifferentials of the infimal nuisance function together with their economic interpretations are given in Section 4. 

\section{Preliminaries}

This section presents some definitions and results that will be used in what follows. 
For a Banach space $X$, the open (resp., closed) unit ball in $X$ will be denoted by $B_X$ (resp., $\bar B_X$). The interior (resp., the closure) of subset $\Omega\subset X$ in the norm topology is abbreviated to $\int \Omega$ (resp., $\overline{\Omega}$). By $\overline{\R}$ we denote the extended real line, i.e., $\overline{\R}=\R\cup\{-\infty, +\infty\}$.

\subsection{Consumer problem}

Let us recall the problem of maximizing utility subject to consumer budget constraint that has been considered in \cite{Penot_2014,Huong_Yao_Yen_2016}. 
The set of \textit{goods} is modeled by a nonempty closed convex cone $X_+$ in a \textit{reflexive} Banach space $X$. The set of \textit{prices} is the positive dual cone of $X_+$
$$Y_+:=\left\{p \in X^*\;:\; \langle p,x \rangle \ge 0\ \, \forall x \in X_+\right\}.$$ 
 We may normalize the prices and assume that the income of the consumer is 1. Then, the \textit{budget map} is the set-valued map $B: Y_+ \rightrightarrows X_+$ associating to each price $p \in Y_+$ the \textit{budget set} 
\begin{equation}\label{budget set}
B(p):=\left\{x\in X_+\;:\; p.x \le 1 \right\}.
\end{equation} For convenience, we put $B(p)=\emptyset$ for every $p\in X^*\setminus Y_+$. In this way, we have a set-valued map $B$ defined on $X^*$ with values in $X$.

We assume that the preferences of the consumer are presented by a function $u:X\rightarrow \overline{\R}$, called the \textit{utility function}. This means that $u(x)\in \R$ for every $x\in X_+$, and a goods bundle $x\in X_+$ is preferred to another one $x' \in X_+$ if and only if $u(x)>u(x')$. For a given price $p\in Y_+$, the problem is to maximize $u(x)$ subject to the constraint $x\in B(p)$. It is written formally as
\begin{equation}\label{initial problem}
\max \left\{u(x)\;:\; x \in B(p)\right\}.
\end{equation}
The \textit{optimal value function} $v:Y_+ \rightarrow \overline{\R}$ of  \eqref{initial problem} is defined by setting
\begin{equation}\label{indirect utility function}
	v(p)=\sup\{u(x):x \in B(p)\} \quad (p\in Y_+).
\end{equation} Since $B(p)=\emptyset$ for all $p\notin Y_+$ and $\sup\emptyset=-\infty$ by an usual convention, one has $v(p)=-\infty$ for all $p\notin Y_+$. In mathematical economics, one often calls $v$ the \textit{indirect utility function}, or the \textit{inverse utility function} of the consumer problem \eqref{initial problem}.
The \textit{demand map} of \eqref{initial problem} is the set-valued map $D: Y_+ \rightrightarrows X_+$ defined by
 \begin{equation}\label{demand_map}
 D(p)=\left\{x\in B(p)\;:\; u(x)=v(p) \right\}  \quad (p \in Y_+).
\end{equation} It is of our convenience to put  $D(p)=\emptyset$ for every $p\in X^*\setminus Y_+$. 

\subsection{Tools from set-valued and variational analysis}

 In this subsection, it is assumed that $X, Y$ are Banach spaces. Let $F: X \rightrightarrows Y$ be a set-valued map. The set ${\rm gph}\, F := \{(x,y) \in X \times Y\,:\, y\in F(x) \}$ is termed the \textit{graph} of $F$. If gph$F$ is closed (resp., convex) in $X\times Y$, then $F$ is said to be \textit{closed} (resp., \textit{convex}). Here, the norm in the product space $X \times Y$ is given by $\| (x, y) \|=\|x\|+\|y\|$ with the first norm in the right-hand side denoting the norm in $X$ and the second one standing for the norm in $Y$. 

One says that $F$ is \textit{Lipschitz-like}, or $F$ has \textit{the Aubin property}, at $(\bar x, \bar y)\in {\rm gph} F$,  if there exists a constant $l>0$ along with neighborhoods $U$ of $\bar x$, $V$ of $\bar y$, such that
$$F(x) \cap V \subset F(x')+l\| x-x' \|\bar{B}_Y \quad \forall x, x' \in U.$$
This fundamental concept was suggested by Aubin \cite{Aubin_1984}. 

If $F: X \rightrightarrows X^*$ be a set-valued map from a Banach space to its dual space, the notation 
\begin{align*}
\Limsup_{x\rightarrow \bar{x}} F(x):=\{ x^*\in X^*
\,:\,  \exists & \mbox{ sequences }x_k \rightarrow \bar{x}  \mbox{ and }  x_k^* \overset{\omega^*}{\rightarrow}  x^* \\
 &\mbox{ with } x_k^*\in F(x_k) \mbox{ for all } k\in \N \},
\end{align*}
where $\N:=\{1, 2, \dots \}$, signifies the \textit{sequential Painlev\'e-Kuratowski upper/outer limit} with respect to the norm topology of $X$ and the weak* topology (that denoted by $\omega^*$) of $X^*$. For a function $\varphi : X \rightarrow \overline \R$ and a set $\Omega \subset X$, the notations $x \overset{\varphi}{\rightarrow} \bar x$ and $x \overset{\Omega}{\rightarrow} \bar x$, respectively, mean 
$x\rightarrow \bar x$ with $\varphi(x) \rightarrow \varphi(\bar x)$ and $x\rightarrow \bar x$ with $x\in \Omega$.

We now recall several basic concepts of generalized differentiation. For more details, the reader is referred to \cite{Mordukhovich_2006a,Mordukhovich_2006b}.

Let $\Omega$ be a subset of $X$. Given $x\in \Omega$ and $\varepsilon \ge 0$, define the set of $\varepsilon$\textit{-normals} to $\Omega$ at $x$ by
\begin{equation}\label{e-normal cone def.}
\widehat{N}_\varepsilon(x; \Omega):=\left\{x^*\in X^* \;:\; \limsup_{u \overset{\Omega}{\rightarrow} x}\dfrac{\langle x^*, u-x \rangle}{\|u-x\|} \le \varepsilon \right\}.
\end{equation}
When $\varepsilon=0$, elements of \eqref{e-normal cone def.} are called \textit{Fr\'echet normals} and their collection, denoted by $\widehat{N}(x; \Omega)$, is the\textit{ prenormal cone} or \textit{Fr\'echet normal cone} to $\Omega$ at $x$. If $x\notin \Omega$, we put $\widehat{N}_\varepsilon(x; \Omega)=\emptyset$ for all $\varepsilon \ge 0$. Consider a vector $\bar{x} \in \Omega$. One say that $x^* \in X^*$ is a \textit{limiting/ Mordukhovich normal} to $\Omega$ at $\bar x$ if there are sequences $\varepsilon_k \downarrow 0, x_k \overset{\Omega}{\rightarrow} \bar x $, and $x_k^* \overset{\omega^*}{\rightarrow}  x^*$ such that $x_k^*\in \widehat{N}_{\varepsilon_k}(x_k; \Omega)$ for all $k\in \N$. The collection of such normals 
\begin{equation*}\label{Mor. normal cone def.}
N(\bar x; \Omega):=\Limsup_{x\rightarrow \bar x \\ \varepsilon \downarrow 0}\widehat{N}_\varepsilon(x; \Omega)
\end{equation*}
is the \textit{limiting/ Mordukhovich normal cone} to $\Omega$ at $x$. Put $N(\bar x; \Omega)=\emptyset$ for $\bar x \notin \Omega$.

Let $F: X \rightrightarrows Y$ be a set-valued map. Given $(\bar x, \bar y) \in X\times Y$ and $\varepsilon \ge 0$, the $\varepsilon$-\textit{coderivative} of $F$ at $(\bar x, \bar y)$ is the set-valued map $\widehat{D}_\varepsilon^*F(\bar x, \bar y): Y^* \rightrightarrows X^*$ with the values
\begin{equation}\label{e-coderivative def.}
\widehat{D}_\varepsilon^*F(\bar x, \bar y)(y^*):=\left\{x^*\in X^* \;:\; (x^*, -y^*)\in \widehat N_\varepsilon((\bar x, \bar y); \gph F)\right\}.
\end{equation}
When $\varepsilon=0$ in \eqref{e-coderivative def.}, this construction is called the \textit{precoderivative} or \textit{Fr\'echet coderivative} of $F$ at $(\bar x, \bar y)$ and is denoted by $\widehat{D}^*F(\bar x, \bar y)$. The \textit{limiting/ Mordukhovich} \textit{coderivative} of $F$ at $(\bar x, \bar y)$ is a set-valued map $D^*F(\bar x, \bar y): Y^* \rightrightarrows X^*$ defined by
\begin{equation*}\label{normal coderivative def.}
D^*F(\bar x, \bar y)(y^*)=\left\{x^*\in X^* \;:\; (x^*,-y^*)\in N((\bar x, \bar y); \gph F)\right\}.
\end{equation*}
It follows from the definition that $\widehat{D}_\varepsilon^*F(x, y)(y^*)=D^*F(\bar x, \bar y)(y^*)=\emptyset$ for all $\varepsilon \ge 0$ and $y^*\in Y^*$ if $(x, y)\notin \gph F$. 
We shall omit $\bar y$ in the coderivative notation if $F(\bar x)=\{\bar y\}$. It follows from above definitions that
\begin{equation}\label{relations between coderivatives}
\widehat{D}^*F(\bar x, \bar y)(y^*)\subset D^*F(\bar x, \bar y)(y^*)
\end{equation}
for any $y^*\in Y^*$.
There are examples showing that inclusions in \eqref{relations between coderivatives} is strict (see \cite[p.~43]{Mordukhovich_2006a}). When \eqref{relations between coderivatives} holds as equality, $F$ is said to be \textit{graphically regular} at $(\bar x, \bar y)$. As shown in the following propositions, the class of graphically regular mappings includes convex set-valued maps and strictly differentiable functions. The reader is referred to \cite[Def.~1.13]{Mordukhovich_2006a} for the definition of strictly differentiable function.
\begin{proposition}\label{code. of convex map}{\rm(See \cite[Prop.~1.37]{Mordukhovich_2006a})} Let $F: X \rightrightarrows Y$ be a convex set-valued map. Then $F$ is graphically regular at $(\bar x, \bar y)\in \gph F$  and one has the codervative representations
\begin{align*}
 \widehat D^*F(\bar x, \bar y)(y^*)&= D^*F(\bar x, \bar y)(y^*)\\
 &=\left\{x^*\in X^*\ :\ \langle x^*, \bar x\rangle-\langle y^*, \bar y\rangle=\max_{(x, y)\in {\rm gph}\, F}[\langle x^*, x\rangle-\langle y^*, y\rangle]\right\}.
 \end{align*}
 \end{proposition}
 \begin{proposition}\label{Coderi. of diffe. functions}{\rm(See \cite[Prop.~1.38]{Mordukhovich_2006a})} 
 Let $f: X \rightarrow Y$ be Fr\'echet differentiable at $\bar x$. Then
 \begin{equation*}
\widehat D^*f(\bar x)(y^*)=\{\nabla f(\bar x)^*y^*\} \quad (\forall y^* \in Y^*),
 \end{equation*}
 where $\nabla f(\bar x)^*:Y^* \rightarrow X^*$ denotes the adjoint operator of the Fr\'echet derivative $\nabla f(\bar x):X \rightarrow Y$. If, moreover, $f$ is strictly differentiable at $\bar x$, then 
 \begin{equation*}
 D^*f(\bar x)(y^*)=\widehat D^*f(\bar x)(y^*)=\{\nabla f(\bar x)^*y^*\} \quad  (\forall y^* \in Y^*),
 \end{equation*}
 and thus $f$ is graphically regular at $\bar x$.
 \end{proposition}

A set $\Omega \subset X$ is said to be \textit{sequentially normally compact} (SNC) at $\bar x \in \Omega$ if for any sequence $(\varepsilon_k, x_k, x^*_k)\in [0, \infty)\times \Omega\times X^*$ satisfying 
\begin{equation*}
\varepsilon_k\downarrow 0,\ \, x_k \rightarrow \bar x,\ \, x_k^*\in\widehat N_{\varepsilon_k}(x_k; \Omega),\ \,  \mbox{and } x^*_k\overset{\omega^*}{\rightarrow} 0
\end{equation*}
one has $\|x_k^*\|\rightarrow 0$ as $k\rightarrow \infty$. A map $F: X \rightrightarrows Y$ is \textit{sequentially normally compact} (SNC) at $(\bar x, \bar y)\in \gph F$ if its graph is SNC at $(\bar x, \bar y)$. Let $\O_1 \subset X_1,\ \O_2\subset X_2$ be subsets of Banach spaces and $(x_1, x_2) \in \O_1\times\O_2$. It follows from definition of SNC of sets and properties of the normal cone to a Cartesian product (see \cite[Proposition 1.2]{Mordukhovich_2006a}) that $\O_1\times\O_2$ is SNC at $(x_1, x_2)$ if $\O_1$ is SNC at $x_1$ and $\O_2$ is SNC at $x_2$.

One says that $\Omega\subset X$ is \textit{locally closed} around $\bar x \in \Omega$ if there is a neighborhood $U$ of $\bar x$ for which $\Omega \cap U$ is closed. If the set $\gph F$ of some set-valued map $F$ is locally closed around $(\bar x, \bar y)\in \gph F$, then $F$ is said to be \textit{locally closed} around $(\bar x, \bar y)$. Clearly, if $\gph F$ is a closed set, then $F$ is locally closed around any point belonging to its graph.

\medskip
For an extended real-valued function $\varphi: X \rightarrow \overline{\R}$, one defines the \textit{epigraph} and \textit{hypograph} of $\varphi$ by  $\epi \varphi=\{(x, \mu)\in X\times \R \,:\, \mu \geq \varphi (x)\}$ and 
$${\rm hypo}\, \varphi=\{(x, \mu)\in X\times \R \,:\, \mu \leq \varphi (x)\}.$$

The  \textit{Fr\'echet subdifferential}, \textit{limiting/Mordukhovich subdifferential}, and the \textit{singular subdifferential} of $\varphi$ at $\bar x\in X$ with $|\varphi (\bar x)|< \infty$ are defined, respectively, by
\begin{equation*}
\widehat \partial\varphi(\bar x):=\left\{x^*\in X \;:\; (x^*, -1)\in \widehat N\big((\bar x, \varphi(\bar x)); \epi \varphi \big)\right\},
\end{equation*}
\begin{equation*}
\partial\varphi(\bar x):=\left\{x^*\in X \;:\; (x^*, -1)\in N\big((\bar x, \varphi(\bar x)); \epi \varphi \big)\right\},
\end{equation*}
and 
\begin{equation*}
\partial^{\infty}\varphi(\bar x):=\left\{x^*\in X \;:\; (x^*, 0)\in N\big((\bar x, \varphi(\bar x)); \epi \varphi \big)\right\}.
\end{equation*}
If $|\varphi (x)|= \infty$, then one puts $\widehat \partial\varphi(\bar x)=\partial\varphi(\bar x)=\partial^{\infty}\varphi(\bar x)=\emptyset$. It follows from above definitions that for all $\bar x \in X$, one has $\widehat \partial\varphi(\bar x)\subset \partial\varphi(\bar x)$. When this holds as equality, one says that $\varphi$ is \textit{lower regular} at $\bar x$.
The \textit{Fr\'echet upper subdifferential}, \textit{limiting/ Mordukhovich upper subdifferential}, and \textit{singular upper subdifferential} of $\varphi$ at $\bar x$ are respectively defined by
\begin{align*}
\widehat \partial^+\varphi(\bar x):=-\widehat \partial(-\varphi)(\bar x)\label{F-upper_subdifferential},\quad
 \partial^+\varphi(\bar x):=-\partial(-\varphi)(\bar x),
\end{align*}
and 
\begin{equation*}
\partial^{\infty,+}\varphi(\bar x):=-\partial^\infty(-\varphi)(\bar x).
\end{equation*}
When $\varphi(\bar x)$ is finite, it follows from definition of the singular subdifferential and the singular upper subdifferential that $\partial^{\infty}\varphi(\bar x)$ and $\partial^{\infty,+}\varphi(\bar x)$ always contain zero, while (see \cite[Corollary 1.81]{Mordukhovich_2006a}) $\partial^{\infty}\varphi(\bar x)=\{0\}$ (and therefore, $\partial^{\infty,+}\varphi(\bar x)=\{0\}$) if $\varphi$ is \textit{locally Lipschitz} around $\bar x$, i.e., there is a neighborhood $U$ of $\bar x$ and a constant $\ell \geq 0$ such that 
\begin{equation*}
|\varphi (x)-\varphi (x')|\le \ell \|x-x'\| \quad (x,\ x' \in U).
\end{equation*}

\section{The Lipchitz-likeness and differentiability properties of the budget map}
We will represent the budget map $B(\cdot):Y_+\rightrightarrows X_+$ as a restriction of the solution map $\widetilde B(\cdot)$ of a parametric generalized equation. Then, we estimate the coderivatives of $\widetilde B(\cdot)$ and obtain a sufficient condition for its Lipschitz-like property. Finally, by the close relationship between the multifunctions $B(\cdot)$ and $\widetilde B(\cdot)$, we obtain  a sufficient condition for the Lipschitz-likeness of the map at point $(\bar p, \bar x)$ in its graph, where $\bar p$ may not belong to the interior of the cone of prices $Y_+$  as well as formula for coderivatives of the budget map $B(\cdot)$.

We will need two theorems from \cite{Mordukhovich_2004_JOGO} on parametric generalized equations, which are recalled now. Let $X, Y, Z$ be Banach spaces.  Consider a \textit{parametric generalized equation}
\begin{equation}\label{para. generalized equa. def.}
0\in f(x, y)+Q(x, y)
\end{equation}
with the decision variable $y$ and the parameter $x$, where $f:X\times Y\rightarrow Z$ is a single-valued map while $Q:X\times Y\rightrightarrows Z$ is a set-valued map. The \textit{solution map} to \eqref{para. generalized equa. def.} is the set-valued map given by
\begin{equation}\label{solution map def.}
S(x):=\left\{y\in Y \;:\; 0\in f(x, y)+Q(x, y)\right\} \quad (x\in X).
\end{equation}

The limiting coderivative of the solution map \eqref{solution map def.} can be estimated or computed in term of the initial data of \eqref{para. generalized equa. def.} by using the following result. 
\begin{theorem}{\rm (See \cite[Theorem 4.1]{Mordukhovich_2004_JOGO})}\label{Thm. 4.1-Mor-2004-JOGO}.
Suppose that $X, Y, Z$ are Asplund spaces, $(\bar x, \bar y)$ satisfy \eqref{para. generalized equa. def.}, $f$ is continuous around $(\bar x, \bar y)$, and $Q$ is locally closed around $(\bar x, \bar y,\bar z)$ with $\bar z:=-f(\bar x, \bar y)$. If $Q$ is SNC at $(\bar x, \bar y, \bar z)$, and
\begin{align}\label{regularity condition}
\left[(x^*, y^*)\in D^*f(\bar x, \bar y)(z^*)\cap\left(-D^*Q(\bar x, \bar y, \bar z)(z^*)\right) \right]
\Longrightarrow (x^*, y^*, z^*)=(0, 0, 0),
\end{align}
then the inclusion
\begin{align}\label{code. of a solution map formula}
\begin{split}
 D^*S(\bar x, \bar y)(y^*) \subset  \big\{x^*\in X^* \;:\;&\exists z^* \in Z^* \mbox{ with }\\
 & (x^*, -y^*)\in D^*f(\bar x, \bar y)(z^*)+D^*Q(\bar x, \bar y, \bar z)(z^*)\big\}
\end{split}
 \end{align} 
 holds for every $y^*\in Y^*$. If, in addition, $f$ is strictly differentiable at $(\bar x, \bar y)$ and $Q$ is graphically regular at $(\bar x, \bar y, \bar z)$, then $S$ is graphically regular at $(\bar x, \bar y)$ and \eqref{code. of a solution map formula} holds as equality.
\end{theorem}

The next theorem states a necessary and sufficient condition for Lipschitz-like property of the solution map \eqref{solution map def.}.

\begin{theorem}{\rm (See \cite[Theorem 4.2]{Mordukhovich_2004_JOGO})}\label{Thm. 4.2-Mor-2004-JOGO}
Let $X, Y, Z$ be Asplund spaces and let $(\bar x, \bar y)$ satisfy \eqref{para. generalized equa. def.}. Suppose that $f$ is strictly differentiable at $(\bar x, \bar y)$ and that $Q$ is locally closed around $(\bar x, \bar y,\bar z)$ with $\bar z:=-f(\bar x, \bar y)$, graphically regular and SNC at this point. If
\begin{equation}\label{regularity condition-a}
\big [(0,0)\in \nabla f(\bar x, \bar y)^*z^*+D^*Q(\bar x, \bar y, \bar z)(z^*) \big]\Longrightarrow z^*=0,\end{equation} then $S$ is Lipschitz-like at $(\bar x, \bar y)$ if and only if 
\begin{equation}\label{Lipschitz-like condition}
\big [(x^*, 0)\in \nabla f(\bar x, \bar y)^*z^*+D^*Q(\bar x, \bar y, \bar z)(z^*) \big]\Longrightarrow [x^*=0,\, z^*=0].
\end{equation}
\end{theorem}

Now, in order to study the budget map in \eqref{budget set}, we define a single-valued map $f:X^*\times X\rightarrow \R\times X$ and a constant set-valued map $Q:X^*\times X\rightrightarrows \R\times X$ by setting 
\begin{equation}\label{f_and_Q}
f(x^*, x)=(\langle x^*, x \rangle -1, -x ),\quad Q(x^*, x)=\R_+\times X_+
\end{equation}
for all $(x^*,x)\in X^*\times X$, where $\R_+:=\{t\in \R\ :\ t\ge 0\}$. Since \begin{equation*}\label{B} B(p)=\{x\in X \;:\; 0\in f(p, x)+Q(p, x) \}\quad (\forall p\in Y_+),\end{equation*} the budget map $B: Y_+ \rightrightarrows X_+$ is the restriction on $Y_+$ of the solution map $\widetilde B: X^* \rightrightarrows X$,
\begin{align}\label{B_widetilde} \widetilde B(x^*)=\{x\in X \;:\; 0\in f(x^*, x)+Q(x^*, x) \}\quad (\forall x^*\in X^*),
\end{align}
of the parametric generalized equation
$0\in f(x^*, x)+Q(x^*, x)$, where $x^* \in X^*$ is a parameter. 
\begin{lemma}\label{code. of single-valued map lemma}
The single-valued map $f$ is strictly differentiable on $X^*\times X$. In addition, for any $(\bar x^*, \bar x) \in X^*\times X$, the coderivative $D^*f(\bar x^*, \bar x):\R\times X^* \rightarrow (X^*\times X)^*$ is the adjoint operator of $\nabla f(\bar x^*, \bar x)$  mapping each $(\lambda, y^*) \in \R\times X^*$ to an element of $(X^*\times X)^*$ which is determined by
\begin{align}\label{code. of single-valued map formula}
\begin{split}
\big\langle D^*f(\bar x^*, \bar x)(\lambda, y^*), (u^*, u)\big\rangle& =
\big\langle \nabla f(\bar x^*, \bar x)^* (\lambda, y^*), (u^*, u) \big \rangle\\
& =\lambda \langle \bar x^*, u\rangle +\lambda \langle u^*,\bar x\rangle-\langle y^*, u \rangle
\end{split}
\end{align}
for all $(u^*, u)\in X^*\times X$. Moreover, if $\bar x$ is a nonzero vector, then the linear operator $D^*f(\bar x^*, \bar x)$ is injective.
\end{lemma}
\begin{proof}
Fix any $(\bar x^*, \bar x) \in X^*\times X$. Consider the operator $T(\bar x^*, \bar x): X^*\times X \rightarrow \R\times X$ defined by
\begin{equation}\label{differ. operator}
T(\bar x^*,\bar x)(u^*, u)=\big( \langle \bar x^*, u \rangle+\langle u^*, \bar x \rangle, -u \big) \ \; ((u^*, u)\in X^*\times X).
\end{equation}
 Clearly, $T(\bar x^*, \bar x)$ is a linear operator. Putting $\gamma=\max\{\|\bar x^*\| +1,\| \bar x \|\}$, for every $(u^*, u)\in X^*\times X$, we have
\begin{align*}
\| T(\bar x^*,\bar x)(u^*, u) \| &=\big| \langle \bar x^*, u \rangle+\langle u^*, \bar x \rangle\big| + \|u\|\\
&\le (\|\bar x^*\| +1)\| u \|+\| u^* \|\| \bar x \|\\
&\le \gamma (\| u^* \|+\| u \|)=\gamma \| (u^*, u) \|.
\end{align*}
Thus, $T(\bar x^*,\bar x)$ is a bounded linear operator. Moreover, 
\begin{align*}\label{Frechet deri. calculate}\begin{array}{rcl}
 & &\displaystyle\lim_{(u^*, u)\rightarrow (\bar x^*, \bar x)}\dfrac{f(u^*, u)-f(\bar x^*,\bar x)- T(\bar x^*,\bar x)((u^*, u)-(\bar x^*,\bar x))}{\|(u^*, u)-(\bar x^*,\bar x)\|}\nonumber\\
& = & \displaystyle\lim_{(u^*, u)\rightarrow (\bar x^*, \bar x)}\Big[
\dfrac{\big ( \langle u^*, u \rangle-1, -u\big)-\big ( \langle \bar x^*,\bar x \rangle-1, -\bar x\big)}{\|(u^*-\bar x^*, u-\bar x)\|}\\
&  & \qquad \qquad  \qquad\quad -\dfrac{\langle \bar x^*, u-\bar x \rangle+ \langle u^*- \bar x^*, \bar x \rangle, -(u-\bar x)}{\|(u^*-\bar x^*, u-\bar x)\|}\Big]\nonumber\\
& = & \displaystyle\lim_{(u^*, u)\rightarrow (\bar x^*, \bar x)}\bigg (\dfrac{\langle u^*-\bar x^*, u-\bar x \rangle}{\|(u^*-\bar x^*, u-\bar x)\|},\ 0 \bigg).
\end{array}
\end{align*}
Since 
\begin{equation*}
\dfrac{|\langle u^*-\bar x^*, u-\bar x \rangle|}{\|(u^*-\bar x^*, u-\bar x)\|}\le \dfrac{\| u^*-\bar x^* \|\| u-\bar x \|}{\| u^*-\bar x^* \|+\|u-\bar x \|}\le \dfrac{\| u^*-\bar x^* \|\| u-\bar x \|}{\| u^*-\bar x^* \|}=\| u-\bar x \|
\end{equation*}
and $\| u-\bar x \|$ converges to $0$ when $u$ tends to $\bar x$, this implies that $f$ is Fr\'echet differentiable at $(\bar x^*, \bar x)$, and we have $\nabla f(\bar x^*, \bar x)=T(\bar x^*, \bar x)$. From \eqref{differ. operator} it follows that the operator $(\bar x^*,\bar x)\mapsto T(\bar x^*,\bar x)$ from $X^*\times X$ to the space of bounded linear operators  $L(X^*\times X, \R\times X)$ is continuous on $X^*\times X$. Hence, $f$ is strictly differentiable at $(\bar x^*, \bar x)$. Therefore, for each $(\lambda, y^*) \in \R\times X^*$, by Proposition \ref{Coderi. of diffe. functions} we have
\begin{equation*}
 \widehat D^*f(\bar x^*, \bar x)(\lambda, y^*)=D^*f(\bar x^*, \bar x)(\lambda, y^*)=\{\nabla f(\bar x^*, \bar x)^* (\lambda, y^*)\},
 \end{equation*}
 where $\nabla f(\bar x^*, \bar x)^*$ denotes the adjoint operator of $\nabla f(\bar x^*, \bar x)$. In addition, by \eqref{differ. operator} and the equality $\nabla f(\bar x^*, \bar x)=T(\bar x^*, \bar x)$ one has 
 \begin{align*}
\big \langle\nabla f(\bar x^*, \bar x)^* (\lambda, y^*), (u^*, u)\big\rangle&=\big \langle(\lambda, y^*), \nabla f(\bar x^*, \bar x) (u^*, u)\big\rangle\\
&=\big \langle(\lambda, y^*),\big( \langle \bar x^*, u \rangle+\langle u^*, \bar x \rangle, -u \big)\big\rangle\\
&=\lambda \langle \bar x^*, u\rangle +\lambda \langle u^*,\bar x\rangle-\langle y^*, u \rangle
 \end{align*}
 for every $(u^*, u)\in X^*\times X$. This establishes formula \eqref{code. of single-valued map formula}.
 
 If $\bar x\neq 0$, then $\nabla f(\bar x^*, \bar x):X^*\times X \rightarrow \R\times X$ is surjective. Indeed, let us show that for any $(\mu, v)\in \R\times X$ there exits $(u^*, u)\in X^*\times X$ with $\nabla f(\bar x^*, \bar x)(u^*, u)=(\mu, v)$. Since $\bar x \neq 0$, by the Hahn-Banach Theorem we can find $u_1^*\in X^*$ such that $\langle  u_1^*, \bar x\rangle=1$. Setting $u^*=(\mu+\langle \bar x^*, v\rangle) u_1^*$ and $u=-v$, we have
 \begin{align*}
 & \ \, u=-v,\ \;
 \langle u^*, \bar x \rangle=\mu+\langle \bar x^*, v \rangle
 \\
 \Leftrightarrow &  \ \,
 \left[ -u=v,\ \;
 \langle \bar x^*, u \rangle+\langle u^*, \bar x \rangle=\mu\right]
\\
 \Leftrightarrow &  \ \,  \big( \langle \bar x^*, u \rangle+\langle u^*, \bar x \rangle, -u \big)=(\mu, v)\\
  \Leftrightarrow &  \ \, T(\bar x^*,\bar x)(u^*, u)=(\mu, v).
 \end{align*}
Hence, $\nabla f(\bar x^*, \bar x)(u^*, u)=(\mu, v)$. Now, since $D^*f(\bar x^*, \bar x)=\nabla f(\bar x^*, \bar x)^*$ and $\nabla f(\bar x^*, \bar x)$ is surjective, we obtain the desired injectivity of $D^*f(\bar x^*, \bar x)$ from \cite[Theorem 4.15]{Rudin-1991} and complete the proof.
 \end{proof} 

 \begin{lemma}\label{code. of set-valued map lemma} The map
$Q:X^*\times X\rightrightarrows \R\times X$ given in \eqref{f_and_Q} is locally closed around and graphically regular at every point of $\gph Q=X^*\times X\times \R_+\times X_+$. For any $(\bar x^*, \bar x, \bar \mu,\bar y)\in\gph Q$, we have
\begin{align}\label{code. of set-valued map formula}
\begin{split}
\widehat D^*Q(\bar x^*, \bar x, \bar \mu,\bar y)(\lambda, y^*)&=D^*Q(\bar x^*, \bar x, \bar \mu,\bar y)(\lambda, y^*)\\
&=\begin{cases}
\{(0, 0)\} & \mbox{if } -\lambda \in N(\bar \mu; \R_+) \mbox{ and } -y^*\in N(\bar y; X_+)\\
\emptyset & \mbox{otherwise}
\end{cases}
\end{split}
\end{align}
with $(\lambda, y^*)\in \R\times X^*$. Moreover, $Q$ is SNC at $(\bar x^*, \bar x, \bar \mu,\bar y) \in \gph Q$ whenever $X_+$ is SNC at $\bar y$. Especially, if $\int X_+\neq \emptyset$, then $Q$ is SNC at every point of its graph.
\end{lemma}
\begin{proof} Since $X_+$ is closed and convex, 
$\gph Q=X^*\times X\times \R_+\times X_+
$
is a closed and convex subset of $X^*\times X\times \R\times X$. Take any $(\bar x^*, \bar x, \bar \mu,\bar y)\in \gph Q$. The closeness of $\gph Q$ implies that $Q$ is locally closed around $(\bar x^*, \bar x, \bar \mu,\bar y)$, while the the convexity of $\gph Q$ and Proposition \ref{code. of convex map} yields that $Q$ is graphically regular at this point, and
\begin{equation*}
\widehat D^*Q(\bar x^*, \bar x, \bar \mu,\bar y)=D^*Q(\bar x^*, \bar x, \bar \mu,\bar y).
\end{equation*}
Take any $(\lambda, y^*)\in \R\times X^*$. By the definition of limiting coderivative and properties of the normal cone to a Cartesian product (see \cite[Proposition 1.2]{Mordukhovich_2006a}), we have
\begin{align*}
&D^*Q(\bar x^*, \bar x, \bar \mu,\bar y)(\lambda, y^*)\\
&=\big \{(x, x^*) \,:\, (x, x^*, -\lambda, -y^*)\in N((\bar x^*, \bar x, \bar \mu,\bar y); \gph Q)\big\}\\
&=\big\{(x, x^*)\,:\, (x, x^*, -\lambda, -y^*)\in N(\bar x^*; X^*)\times N(\bar x; X)\times N(\bar \mu; \R_+)\times N(\bar y; X_+)\big\}\\
&=\begin{cases}
\{(0, 0)\} & \mbox{if } -\lambda \in N(\bar \mu; \R_+)\ \, \mbox{and}\ -y^*\in N(\bar y; X_+)\\
\emptyset & \mbox{otherwise}.
\end{cases}
\end{align*}
Thus, formula \eqref{code. of set-valued map formula} is valid.

Next, suppose that $(\bar x^*, \bar x, \bar \mu,\bar y) \in \gph Q$. By  \cite[Prop.~1.25 and Theorem~1.26]{Mordukhovich_2006a}, any convex set with nonempty interior is SNC at every point belonging to it. Hence,  $X^*$, $X$, $\R_+$ are SNC at $\bar x^*, \bar x, \bar \mu$, respectively. If, in addition, $X_+$ is SNC at $\bar y$ (which is automatically satisfied if $\int X_+\neq \emptyset$), then $\gph Q=X^*\times X\times \R_+\times X_+$ is SNC at $(\bar x^*, \bar x, \bar \mu,\bar y)$. So, $Q$ is SNC at $(\bar x^*, \bar x, \bar \mu,\bar y)$ whenever $X_+$ is SNC at $\bar y$.

The proof is complete.
\end{proof}
\begin{lemma} 
Let $\bar x^* \in X^*,\ \bar x \in \widetilde B(\bar x^*)$, where $\widetilde B$ is the map in \eqref{B_widetilde}, be such that $\bar x \neq 0$ and $X_+$ is SNC at $\bar x$. Then $\widetilde B: X^*\rightrightarrows X$ is graphically regular at $(\bar x^*, \bar x)$. Moreover, for every $x^*\in X^*$, 
\begin{align}\label{code. of solution  map formula}
\begin{split}
&\widehat D^*\widetilde B(\bar x^*, \bar x)(x^*)= D^*\widetilde B(\bar x^*, \bar x)(x^*)\\
&=\begin{cases}
\left\{\lambda \bar x \;:\; \lambda\geq 0,\; x^*+\lambda \bar x^*\in -N(\bar x; X_+)\right\}&\mbox{if }\langle \bar x^*,\bar x \rangle=1 \\
\{0\}&\mbox{if }\langle \bar x^*, \bar x \rangle<1,\; x^*\in -N(\bar x; X_+)\\
\emptyset &\mbox{if }\langle \bar x^*,\bar x \rangle<1,\; x^*\notin -N(\bar x; X_+).
\end{cases}
\end{split}
\end{align}
\end{lemma}
\begin{proof}
Suppose that $\bar x^* \in X^*$, $\bar x\in \widetilde B(\bar x^*)\setminus\{0\}$, and $X_+$ is SNC at $\bar x$. Since $X$ is reflexive, so are $X^*$, $X^*\times X$, and $\R\times X$. Hence, these spaces are Asplund and $(X^*)^*=X$, $(X^*\times X)^*= X\times X^*$, $\R\times X)^*=\R\times X^*$. The proof is divided into two steps. In the first step, we will apply Theorem \ref{Thm. 4.1-Mor-2004-JOGO} to show that, for every $x^*\in X^*$, 
\begin{align}\label{code. of solution  map step 1}
\begin{split}
D^*\widetilde B(\bar x^*, \bar x)(x^*)
=\big\{ x\in X \,:\, \, \exists \lambda&\in -N(1-\langle \bar x^*, \bar x \rangle; \R_+),\; \exists y^*\in -N(\bar x; X_+)\\& \quad \mbox{ such that }
\, (x, -x^*)= \nabla f(\bar x^*, \bar x)^*(\lambda, y^*) \big\}.
\end{split}
\end{align}
In the second step, we will prove that the set $A(x^*)$  in the right-hand side of \eqref{code. of solution  map step 1} can be computed by the formula
\begin{align}\label{code. of solution  map step 2}
A(x^*)=\begin{cases}
\left\{\lambda \bar x \;:\; \lambda\geq 0,\; x^*+\lambda \bar x^*\in -N(\bar x; X_+)\right\}&\mbox{if }\langle \bar x^*,\bar x \rangle=1 \\
\{0\}&\mbox{if }\langle \bar x^*, \bar x \rangle<1,\; x^*\in -N(\bar x; X_+)\\
\emptyset &\mbox{if }\langle \bar x^*,\bar x \rangle<1,\; x^*\notin -N(\bar x; X_+).
\end{cases}
\end{align}

{\sc Step 1.} From Lemma \ref{code. of single-valued map lemma}, $f$ is strictly differentiable at $(\bar x^*, \bar x)$ and its coderivative is given by \eqref{code. of single-valued map formula}.
Besides, since $\bar x\in \widetilde B(\bar x^*)$, $\bar x\in X_+$ and $1-\langle \bar x^*, \bar x \rangle \in \R_+$. Thus, $(\bar x^*, \bar x, 1-\langle \bar x^*, \bar x \rangle, \bar x )\in \gph Q$. Lemma \ref{code. of set-valued map lemma} implies that $Q$ is locally closed around $(\bar x^*, \bar x, 1-\langle \bar x^*, \bar x \rangle, \bar x )$, graphically regular and SNC at this point if $X_+$ is SNC at $\bar x$. So, if  condition \eqref{regularity condition} is satisfied, then $\widetilde B$ is graphically regular at $(\bar x^*, \bar x)$ and we can estimate the coderivatives of $\widetilde B$ at $(\bar x^*, \bar x)$ by formula \eqref{code. of a solution map formula} when it holds as equality. To check \eqref{regularity condition}, we fix any  $(x, x^*, \lambda, y^*)\in X\times X^* \times \R\times X^*$ satisfying the inclusion
\begin{equation}\label{regularity for the budget map}
(x, x^*)\in D^*f(\bar x^*, \bar x)(\lambda, y^*) \bigcap \Big(-D^*Q(\bar x^*, \bar x, 1-\langle \bar x^*, \bar x \rangle, \bar x )(\lambda, y^*)\Big).
\end{equation} 
By Lemma \ref{code. of set-valued map lemma}, one has 
$D^*Q(\bar x^*, \bar x, 1-\langle \bar x^*, \bar x \rangle, \bar x )(\lambda, y^*) \subset \{(0, 0)\}\subset X\times X^*.$
Combining this with \eqref{regularity for the budget map} and noting that $D^*f(\bar x^*, \bar x)(\lambda, y^*)$ is a singleton, we obtain $(x, x^*)=(0, 0)$ and $D^*f(\bar x^*, \bar x)(\lambda, y^*) =(0, 0)$. In addition, since $\bar x\neq 0$,  $D^*f(\bar x^*, \bar x)$ is injective by Lemma \ref{code. of single-valued map lemma}. So we have  $(\lambda, y^*)=(0, 0)\in \R\times X^*$.  Therefore, if $(x, x^*, \lambda, y^*)\in X\times X^* \times \R\times X^*$ satisfies \eqref{regularity for the budget map}, then $(x, x^*, \mu, y^*)=(0, 0, 0, 0)$. Thus, condition \eqref{regularity condition} is fulfilled.

By the second assertion of Theorem \ref{Thm. 4.1-Mor-2004-JOGO}, the set-valued map $\widetilde B: X^*\rightrightarrows X$ is graphically regular at $(\bar x^*, \bar x)$, and the coderivative $D^*\widetilde B(\bar x^*, \bar x): X^*\rightrightarrows X$ maps each $x^*\in X^*$ to the set
\begin{align*}
D^*\widetilde B(\bar x^*, \bar x)(x^*)=\big\{& x\in X \;:\; \exists (\lambda, y^*)\in \R\times X^*\mbox{ such that }\\
\ &(x, -x^*)\in D^*f(\bar x^*, \bar x)(\lambda, y^*)+ D^*Q(\bar x^*, \bar x, 1-\langle \bar x^*, \bar x \rangle, \bar x )(\lambda, y^*) \big\}.
\end{align*}
In addition, by \eqref{code. of set-valued map formula},
\begin{align*}
& D^*Q(\bar x^*, \bar x, 1-\langle \bar x^*, \bar x \rangle, \bar x )(\lambda, y^*)\\
&=\begin{cases}
\{(0, 0)\}  & \mbox{if } -\lambda \in N(1-\langle \bar x^*, \bar x \rangle; \R_+)\,\mbox{ and }\, -y^*\in N(\bar x; X_+)\\
\emptyset & \mbox{otherwise}.
\end{cases}
\end{align*}
Hence, we obtain formula \eqref{code. of solution  map step 1}.

{\sc Step 2.} For any $x^*\in X^*,\ x\in X,\ \lambda \in \R,\ y^*\in X^*$, the equality 
\begin{align}\label{abc}
 (x, -x^*)= \nabla f(\bar x^*, \bar x)^*(\lambda, y^*)
\end{align} holds if and only if
$$\langle (x, -x^*), (u^*, u)\rangle=\langle \nabla f(\bar x^*, \bar x)^*(\lambda, y^*), (u^*, u)\rangle\quad (\forall u^*\in X^*, \forall u \in X).$$
So,  by \eqref{code. of single-valued map formula}, \eqref{abc} means that
$$\langle (x-\lambda\bar x, y^*-x^*-\lambda\bar x^*), (u^*, u)\rangle=0 \quad (\forall u^*\in X^*, \forall u \in X).$$
Clearly, the latter is equivalent to the following system
\begin{align}\label{def}
 x=\lambda \bar x,\quad y^*=x^*+\lambda\bar x^*. 
\end{align}
Now, fix an $x^*\in X^*$. If $\langle \bar x^*, \bar x \rangle<1$, then $-N(1-\langle \bar x^*, \bar x \rangle; \R_+)=\{0\}$. Hence,  $x\in A(x^*)$ if and only if \eqref{abc} holds for $\lambda=0$ and for some $y^* \in-N(\bar x; X_+)$. From this property and \eqref{def} it follows that $A(x^*)=\{0\}$  when $x^* \in -N(\bar x; X_+)$ and $A(x^*)=\emptyset$ when $x^* \notin -N(\bar x; X_+)$. If $\langle \bar x^*, \bar x \rangle = 1$, then $-N(1-\langle \bar x^*, \bar x \rangle; \R_+)=\R_+$. Therefore,  $x\in A(x^*)$ if and only if \eqref{abc} holds for some $\lambda\geq 0$ and $y^* \in-N(\bar x; X_+)$.  So, $x\in A(x^*)$ if and only if $x$ fulfills \eqref{def} with $\lambda \in \R_+$ satisfying the condition $x^*+\lambda \bar x^*\in -N(\bar x; X_+)$. Formula \eqref{code. of solution  map step 2} has been obtained.

The proof is complete.
\end{proof}

In our preceding paper \cite{Huong_Yao_Yen_2016}, a H\"older-Lipschitz property of the demand map $D(\cdot)$ was obtained by using the Lipschitz-like property of the budget map $B(\cdot)$ at point $(\bar p, \bar x) \in$ $\gph B$ with $\bar p$ being an interior point of the cone of prices $Y_+$. Now, we will show that if $\bar x \neq 0$ and $X_+$ is SNC at $\bar x$, then we can get the Lipschitz-like property of $B(\cdot)$  without imposing the condition $\bar p \in \mbox{int }Y_+$. Hence, Theorem 4.4 in \cite{Huong_Yao_Yen_2016} can be extended to the case where $\bar p$ may belong to the boundary of $Y_+$.

\begin{theorem}
	Assume that $\bar x^* \in X^*,\, \bar x \in \widetilde B(\bar x^*)\setminus\{0\}$, and $X_+$ is SNC at $\bar x$. Then, the solution map $\widetilde B(\cdot)$ is Lipschitz-like at $(\bar x^*, \bar x)\in \gph \widetilde B$. Consequently, if $\bar p \in Y_+$, $\bar x\in B(\bar p)\setminus\{0\}$, and $X_+$ is SNC at $\bar x$, then the budget map $B(\cdot)$ is Lipschitz-like at $(\bar p, \bar x)$ in the sense that there exist a neighborhood $U$ of $\bar p$, a neighborhood $V$ of $\bar x$, and a constant $\ell > 0$ satisfying
	\begin{equation}\label{L-l_B}
	B(p)\cap V\subset B(p')+\ell\|p-p'\|\overline{B}_X \quad \forall p, p'\in U\cap Y_+.
	\end{equation}
\end{theorem}
\begin{proof}
	Suppose that $\bar x^*\in X^*$, $\bar x \in \widetilde B(\bar x^*)\setminus\{0\}$, and $X_+$ is SNC at $\bar x$. We apply Theorem \ref{Thm. 4.2-Mor-2004-JOGO} with $f,\, Q$ being given by \eqref{f_and_Q}, $(\bar x^*, \bar x)$ and $\widetilde B(\cdot)$ playing the roles of $(\bar x, \bar y)$ and $S(\cdot)$, respectively. Thus, we are dealing with the generalized equation  $0\in f(x^*, x)+Q(x^*, x)$ appeared in \eqref{B_widetilde}. 
	
	Since $X$ is reflexive, so are $X^*$ and $\R\times X$. Hence,  $X^*$, $X$, and $\R\times X$ are Asplund spaces. Besides, it follows from Lemma \ref{code. of single-valued map lemma} that $f$ is strictly differentiable at $(\bar x^*, \bar x)$. Moreover, as $\bar x\in \widetilde B(\bar x^*)$, one has $\bar x\in X_+$ and $1-\langle \bar x^*, \bar x \rangle \in \R_+$. Thus, for $\bar z:=(1-\langle \bar x^*, \bar x \rangle, \bar x )$, one has $(\bar x^*, \bar x, \bar z)\in \gph Q=X^*\times X\times\R_+\times X_+$.  Since $X_+$ is SNC at $\bar x$, Lemma \ref{code. of set-valued map lemma} assures that $Q$ is locally closed around $(\bar x^*, \bar x, \bar z)$, graphically regular and SNC at this point. Therefore, it remains to check \eqref{regularity condition-a} and \eqref{Lipschitz-like condition}.
	
	 Condition \eqref{regularity condition-a} requires 
	\begin{equation*}
	\big[(0,0)\in \nabla f(\bar x^*, \bar x)^*(\lambda, y^*)+D^*Q(\bar x^*, \bar x, \bar z)(\lambda, y^*) \big]\Longrightarrow (\lambda, y^*)=(0, 0).
	\end{equation*}
	Fix any $(\lambda, y^*)\in\R\times X^*$ with the property
	\begin{equation}\label{regularity condition-a_left side}
(0,0)\in \nabla f(\bar x^*, \bar x)^*(\lambda, y^*)+D^*Q(\bar x^*, \bar x, \bar z)(\lambda, y^*).
	\end{equation}
	 By \eqref{code. of set-valued map formula}, the inclusion $D^*Q(\bar x^*, \bar x, \bar z)(\lambda, y^*) \subset \{(0,0)\}$ is valid. Hence,   \eqref{regularity condition-a_left side} implies that $\nabla f(\bar x^*, \bar x)^*(\lambda, y^*)=(0,0)$. On one hand, combining this with  \eqref{code. of single-valued map formula}, we get $D^*f(\bar x^*, \bar x)(\lambda, y^*)=\nabla f(\bar x^*, \bar x)^*(\lambda, y^*)=(0,0)$. On the other hand, since $\bar x \neq 0$, $D^*f(\bar x^*, \bar x)$ is injective by Lemma \ref{code. of single-valued map lemma}. It follows that $(\lambda, y^*)=(0,0)$. Thus, condition \eqref{regularity condition-a} is satisfied. 
	
	Condition \eqref{Lipschitz-like condition} is the following
	\begin{equation*}
	\big [(x, 0)\in \nabla f(\bar x^*, \bar x)^*(\lambda, y^*)+D^*Q(\bar x^*, \bar x, \bar z)(\lambda, y^*) \big]\; \Longrightarrow\; (x, \lambda, y^*)=(0, 0, 0).
	\end{equation*}
	Let $(x, \lambda, y^*)\times\R\times X^*$ be such that
	\begin{equation*}\label{regularity condition-b_left side}
	(x, 0)\in \nabla f(\bar x^*, \bar x)^*(\lambda, y^*)+D^*Q(\bar x^*, \bar x, \bar z)(\lambda, y^*).
	\end{equation*}
In combination with \eqref{code. of set-valued map formula}, this yields $\lambda \in  -N(1-\langle \bar x^*, \bar x \rangle; \R_+)$, $y^*\in -N(\bar x; X_+)$, and $(x, 0)= \nabla f(\bar x^*, \bar x)^*(\lambda, y^*)$. Hence, using formula \eqref{code. of solution  map step 1} for $x^*=0$, we obtain $x\in D^*\widetilde B(\bar x^*, \bar x)(0)$. Let us show that 
	\begin{equation}\label{qualification-1}
	D^*\widetilde B(\bar x^*, \bar x)(0)=\{0\}.
	\end{equation}
	Since $X_+$ is a convex cone, $N(\bar x; X_+)$ is the normal cone in the sense of convex analysis. Clearly, $0\in -N(\bar x; X_+)$. As $\bar x\in \widetilde{B}( \bar x^*)$, one has $\langle \bar x^*, \bar x \rangle \le 1$. In case $\langle \bar x^*, \bar x \rangle<1$, by formula \eqref{code. of solution  map formula}, one gets \eqref{qualification-1}. In case $\langle \bar x^*, \bar x \rangle=1$, \eqref{code. of solution  map formula} implies 
	\begin{equation}\label{coderivative of tilde B at 0}
		D^*\widetilde B(\bar x^*, \bar x)(0)=\{t\bar x \,:\, t\geq 0,\; t\bar x^* \in -N(\bar x; X_+)\}.
	\end{equation}
By \eqref{coderivative of tilde B at 0}, if $D^*\widetilde B(\bar x^*, \bar x)(0)$ contains a nonzero vector, then the latter must have the form $t\bar x $ with $t> 0$ and $-t\bar x^* \in N(\bar x; X_+)$. For $\tilde{x}:=0 \in X_+$, we have $$\langle -t\bar x^*, \tilde{x}-\bar x\rangle=t\langle \bar x^*, \bar x \rangle=t >0,$$ which contradicts the inclusion $-t\bar x^* \in N(\bar x; X_+)$.  So, \eqref{qualification-1} is valid. Since $x$ belongs to $D^*\widetilde B(\bar x^*, \bar x)(0)$, from \eqref{qualification-1} one has $x=0$. So, as $\nabla f(\bar x^*, \bar x)^*(\lambda, y^*)=(x, 0)$, we get $\nabla f(\bar x^*, \bar x)^*(\lambda, y^*)=(0, 0)$. Hence, by~\eqref{code. of single-valued map formula}, $$D^*f(\bar x^*, \bar x)(\lambda, y^*)=\nabla f(\bar x^*, \bar x)^*(\lambda, y^*)=(0,0).$$ Since $D^*f(\bar x^*, \bar x)$ is injective by Lemma~\ref{code. of single-valued map lemma}, this yields  $(\lambda,y^*)=(0,0)$. We have shown that $(x, \lambda, y^*)=(0, 0, 0)$; thus  \eqref{Lipschitz-like condition} is fulfilled.
 
The above analysis allows us to invoke Theorem \ref{Thm. 4.2-Mor-2004-JOGO} to assert that $\widetilde B: X^*\rightrightarrows X$ is Lipschitz-like at $(\bar x^*, \bar x)$. Now, take any $\bar p \in Y_+$ and $\bar x\in B(\bar p)\setminus\{0\}$. Since $\bar x \in B(\bar p)=\widetilde B(\bar p)$, one has $\bar x \in \widetilde B(\bar p)$.  Thus, if $X_+$ is SNC at $\bar x$, then $\widetilde B$ is Lipschitz-like at $(\bar p, \bar x)$. Hence, there exist a neighborhood $U$ of $\bar p$, a neighborhood $V$ of $\bar x$, and a constant $\ell > 0$ satisfying
	\begin{equation}\label{L-l_widetileB}
	\widetilde B(p)\cap V\subset \widetilde B(p')+\ell\|p-p'\|\overline{B}_X \quad \forall p, p'\in U.
	\end{equation}
	Remembering that $\widetilde B(p)=B(p)$ for all $p\in U\cap Y_+$, we obtain \eqref{L-l_B} from \eqref{L-l_widetileB} and complete the proof.
\end{proof}

Under some mild conditions, we can have exact formulas for both Fr\'echet and limiting coderivatives of the budget map.

\begin{theorem}\label{code. of budget map}
Suppose that $\bar p \in \int Y_+,\; \bar x \in B(\bar p)\setminus\{0\}$, and $X_+$ is SNC at $\bar x$. Then the budget map $B: Y_+\rightrightarrows X_+$ is graphically regular at $(\bar p, \bar x)$. Moreover, for every $x^*\in X^*$, one has
\begin{align}\label{code. of budget map formula}
\begin{split}
& \widehat D^*B(\bar p, \bar x)(x^*)=D^*B(\bar p, \bar x)(x^*)\\
&=\begin{cases}
\left\{\lambda \bar x \;:\; \lambda\geq 0,\; x^*+\lambda \bar p\in -N(\bar x; X_+)\right\}&\mbox{if }\langle \bar p, \bar x \rangle=1 \\
\{0\}&\mbox{if }\langle \bar p, \bar x \rangle<1,\, x^*\in -N(\bar x; X_+)\\
\emptyset &\mbox{if }\langle \bar p, \bar x \rangle<1,\, x^*\notin -N(\bar x; X_+).
\end{cases}
\end{split}
\end{align}
\end{theorem}
\begin{proof} Since  $\bar p \in \int Y_+$, there exists an open set $U$ in the norm topology of $X^*$ such that $\bar p\in U\subset Y_+$. Then $B(p)= \widetilde B(p)$ for all $p\in U$. It follows that 
\begin{equation}\label{local_coincidence}
(\gph B)\cap (U\times X) =(\gph \widetilde{B})\cap (U\times X).
\end{equation}
By the definitions of Fr\'echet and limiting normal cones, the cones $\widehat N((\bar p, \bar x); \gph B)$ and $N((\bar p, \bar x); \gph B)$ (resp., $\widehat N((\bar p, \bar x); \gph \widetilde B)$ and 
$N((\bar p, \bar x);\gph \widetilde B)$) are defined just by vectors of $\gph B$ (resp., of $\gph \widetilde B$) in a neighborhood $W$ of $(\bar p, \bar x)$. Choosing $W=U\times X$, from \eqref{local_coincidence} we can deduce that 
$\widehat N((\bar p, \bar x); \gph B)=\widehat N((\bar p, \bar x); \gph \widetilde B)$ and $N((\bar p, \bar x); \gph B)=N((\bar p, \bar x); \gph \widetilde B)$. Consequently, for any $x^*\in X^*$,
$$
\widehat D^* B(\bar p, \bar x)(x^*)=\widehat D^*\widetilde B(\bar p, \bar x)(x^*)
$$
and 
\begin{equation*}\label{eee}
D^* B(\bar p, \bar x)(x^*)=D^*\widetilde B(\bar p, \bar x)(x^*).
\end{equation*}
Hence, letting $\bar p$ play the role of $\bar x^*$ in \eqref{code. of solution  map formula}, we obtain formula \eqref{code. of budget map formula} from the latter.

The proof is complete.
\end{proof}

If $\langle \bar p, \bar x \rangle=1$ then, for any $x^* \in X^*$, using \eqref{code. of budget map formula} one can compute the coderivative values $\hat D^*B(\bar p, \bar x)(x^*)$ and $D^*B(\bar p, \bar x)(x^*)$ via the set $\{\lambda\geq 0 \;:\; x^*+\lambda \bar p\in -N(\bar x; X_+)\}$  of real numbers. 
The forthcoming lemma, which will be used intensively in Section~4, describes explicitly the latter set in a situation where $\bar x\in D(\bar p)$ and $x^*=-\nabla u(\bar x)$.

\begin{lemma}\label{neccessary condition lemma}
Suppose that $(\bar p, \bar x)\in \gph D$ and $u$ is Fr\'echet differentiable at $\bar x$. Then 
\begin{equation}\label{neccessary condition}
\{\lambda\geq 0 \;:\; \lambda \bar p\in \nabla u(\bar x) -N(\bar x; X_+)\}=\{\langle \nabla u(\bar x), \bar x\rangle\}
\end{equation}
when $\langle \bar p, \bar x \rangle=1$, and $\nabla u(\bar x) \in N(\bar x; X_+)$ when $\langle \bar p, \bar x \rangle<1$.
\end{lemma}
\begin{proof} Let $\bar p, \bar x, u$ satisfy our assumptions. 

Suppose that $\langle \bar p, \bar x \rangle=1$. We first establish the inclusion ``$\subset$'' in \eqref{neccessary condition}, and then show that the set on the left-hand side, which will be denoted by $\Lambda$, is nonempty. 

Given any $\lambda\in\Lambda$, we have $\lambda \bar p = \nabla u(\bar x)-z^*$ for some $z^*\in N(\bar x, X_+)$. Combining the last equality with the condition $\langle \bar p, \bar x\rangle =1$ gives $\lambda= \langle \nabla u(\bar x), \bar x\rangle - \langle z^*, \bar x\rangle$. In addition, since $z^*\in N(\bar x, X_+)$ and $X_+$ is convex, $\langle z^*,x - \bar x\rangle \leq 0$ for all $x \in X_+$. Besides, as $X_+$ is a nonempty closed cone, $x_1:=0$ and $x_2:=2\bar x$ belong to $X_+$. Therefore, $-\langle x^*,\bar x\rangle = \langle z^*,x_1 - \bar x\rangle \leq 0$ and $\langle z^*,\bar x\rangle = \langle z^*,x_2 - \bar x\rangle \leq 0$. Hence, we must have $\langle z^*, \bar x\rangle =0$. Thus, $\lambda= \langle \nabla u(\bar x), \bar x\rangle$ and we obtain the inclusion ``$\subset$'' in \eqref{neccessary condition}.

Next, fix an arbitrary $x \in B(\bar p)$. As $B(\bar p)$ is convex, $x_t:=tx+(1-t)\bar x$ belongs to $B(\bar p)$ for all $t\in (0, 1)$. Since $\bar x \in D(\bar p)$, $u(\bar x)\geq u(x_t)$ for all $t\in (0, 1)$. Thus, $[u(\bar x +t(x-\bar x))-u(\bar x)]/t \leq 0$ for all $t\in (0, 1)$. Letting $t\to 0^+$ and using the Fr\'echet differentiability of $u$ at $\bar x$, we obtain $\langle \nabla u(\bar x),x-\bar x\rangle \leq 0$. Since the latter holds for any $x\in B(\bar p)$, by the convexity of $B(\bar p)$ we have $\nabla u(\bar x)\in N(\bar x, B(\bar p))$. Now, as $B(\bar p)=X_+\cap \Omega$ with $\Omega:=\{x\in X \;:\; \langle \bar p, x \rangle\leq 1 \}$ and $0 \in X_+\cap \int \Omega$, applying the fundamental intersection rule of convex analysis \cite[Proposition~1, p.~205]{Ioffe_Tihomirov_1979}, we have
\begin{equation}\label{intersection}
N(\bar x, B(\bar p))=N(\bar x, X_+)+N(\bar x, \Omega).
\end{equation}
Clearly, $N(\bar x, \Omega)=\R_+\bar p$ when $\langle \bar p, \bar x \rangle = 1$. So, the inclusion $\nabla u(\bar x)\in N(\bar x, B(\bar p))$ and \eqref{intersection} imply the existence of $\lambda \geq 0$ satisfying $\nabla u(\bar x)\in N(\bar x, X_+)+\lambda \bar p$. This shows that $\Lambda\neq\emptyset$.

 If $\langle \bar p, \bar x \rangle<1$, then $N(\bar x, \Omega)=\{0\}$. Thus, the inclusion $\nabla u(\bar x)\in N(\bar x, B(\bar p))$ and~\eqref{intersection} yield $\nabla u(\bar x) \in N(\bar x; X_+)$.
 
The proof is complete. \end{proof}

\section{Subgradients of the function $-v$}
Following \cite{Mordukhovich_Nam_Yen_2009}, we consider the \textit{parametric optimization problem}
\begin{equation}\label{MNY_problem}
\min\{\varphi(x, y) \;:\;  y\in G(x)\},
\end{equation}
where $\varphi:X\times Y \rightarrow \overline{\R}$ is a \textit{cost} function, $G:X\rightrightarrows Y$ is a \textit{constraint} set-valued map between Banach spaces. The \textit{marginal function} $\mu (\cdot):X\rightarrow \overline{\R}$ and the \textit{solution map} $M(\cdot):X\rightrightarrows Y$ of this problem are defined, respectively, by
\begin{equation*}
\mu(x):=\inf\{\varphi(x, y) \;:\; y\in G(x)\},
\end{equation*}
and
\begin{equation*}
M(x):=\left\{y\in G(x) \;:\; \mu(x)=\varphi(x, y) \right\}.
\end{equation*}

\subsection{Fr\'echet subgradients}

The following theorem gives an upper estimate for the Fr\'echet subdifferential of the general marginal function $\mu(\cdot)$ at a given point $\bar x$ via the Fr\'echet coderivative of the constraint mapping $G$ and the Fr\'echet upper subdifferential of the value function $\varphi$.

\begin{theorem}\label{Frechet upper subdif. MNY09}{\rm(See \cite[Theorem 1]{Mordukhovich_Nam_Yen_2009})}
Let $\bar x$ be such that $M(\bar x)\neq \emptyset$ and $|\mu(\bar x)| \neq \infty $, and let $\bar y \in M(\bar x)$ be such that $\widehat \partial^+\varphi(\bar x, \bar y)\neq \emptyset$. Then 
\begin{equation}\label{inclusion}
\widehat \partial\mu(\bar x)\subset\bigcap_{(x^*, y^*)\in \widehat \partial^+\varphi(\bar x, \bar y)}\big[x^*+\widehat D^*G(\bar x, \bar y)(y^*)\big].
\end{equation}
\end{theorem}

To recall the sufficient conditions of \cite{Mordukhovich_Nam_Yen_2009} for the inclusion in \eqref{inclusion} to hold as equality, we need the following definitions. Let $X, Y$ be Banach spaces, $D\subset X$.
A map $h:D \rightarrow Y$ is said to be \textit{locally upper Lipschitzian} at $\bar x \in D$ if there are a neighborhood $U$ of $\bar x$ and a constant $\ell > 0$ such that
$$
\|h(x)-h(\bar x)\| \leq \ell\|x-\bar x\|\quad  \mbox{whenever}\ \; x\in D\cap U.
$$
We say that a set-valued map $F:D\rightrightarrows Y$ admits \textit{a local upper Lipschitzian selection} at $(\bar x, \bar y)\in \gph F$ if there is a map $h:D\rightarrow Y$, such that $h$ is locally upper Lipschitzian at $\bar x$,  $h(\bar x)=\bar y$, and $h(x)\in F(x)$ for all $x\in D$ in a neighborhood of $\bar x$. 

\begin{theorem}{\rm (See \cite[Theorem 2]{Mordukhovich_Nam_Yen_2009})}\label{Frechet subdif. MNY09}
In addition to the assumptions of Theorem~\ref{Frechet upper subdif. MNY09}, suppose that $\varphi$ is Fr\'echet differentiable at $(\bar x, \bar y)$, and the map $M:\dom G\rightrightarrows Y$ admits a local upper Lipschitzian selection at $(\bar x, \bar y)$. Then 
\begin{equation*}
\widehat \partial \mu(\bar x)=x^*+\widehat D^*G(\bar x, \bar y)(y^*),\ \; \mbox{where}\ \, (x^*, y^*):=\nabla\varphi(\bar x, \bar y).
\end{equation*} 
\end{theorem}

Fr\'echet subgradients of the function $-v$ can be computed by next theorem.

\begin{theorem}\label{Frechet subdif. HYY17}
Let $\bar p\in \int Y_+$ and $\bar x \in D(\bar p) \setminus \{0\}$ be such that $D(\bar p)\neq \emptyset$, $X_+$ is SNC at $\bar x$, and $\widehat \partial u(\bar x)\neq \emptyset$. The following assertions hold:
\begin{itemize}
\item [{\rm (i)}] If $\langle \bar p, \bar x\rangle = 1$, then 
\begin{equation}\label{Fre. subgradients-upper inclusion-a}
\widehat \partial(-v)(\bar p)\subset \displaystyle\bigcap_{x^*\in -\widehat \partial u(\bar x)}\{\lambda \bar x  \;:\; \lambda\geq 0,\; x^*+\lambda \bar p \in -N(\bar x, X_+)\};
\end{equation}
\item[{\rm (ii)}] If $\langle \bar p, \bar x\rangle <1$, then 
\begin{equation}\label{Fre. subgradients-upper inclusion-b}
\widehat \partial(-v)(\bar p)\subset \{0\};
\end{equation}
\item[{\rm (iii)}] If $\langle \bar p, \bar x\rangle <1$ and $\,\widehat \partial u(\bar x)\setminus N(\bar x, X_+) \neq \emptyset$, then
\begin{equation}\label{Fre. subgradients-upper inclusion-c}
\widehat \partial(-v)(\bar p)=\emptyset;
\end{equation}
\item[{\rm (iv)}] If $u$ is Fr\'echet differentiable at $\bar x$, and the map $D:\dom B \rightrightarrows X_+$ admits a local upper Lipschitzian selection at $(\bar p, \bar x)$, then
\begin{align}\label{Fre. subgradients-upper equality}
\widehat \partial (-v)(\bar p)=\begin{cases}
\{\langle \nabla u(\bar x), \bar x\rangle\bar x\} \! &\! \mbox{if}\; \langle \bar p, \bar x\rangle = 1\\
\{0\}&\mbox{if}\; \langle \bar p, \bar x\rangle < 1.
\end{cases}
\end{align}
\end{itemize}
\end{theorem}
\begin{proof} To transform the consumer problem in \eqref{initial problem} to the minimization problem \eqref{MNY_problem}, we let $X^*$ (resp., $X$) play the role of $X$ (resp., $Y$). Put $\varphi(x^*, x)=-u(x)$ for $(x^*, x)\in X^*\times X$. Besides, let $G(x^*)=\left\{x\in X_+:\langle x^*,x\rangle \le 1 \right\}$ for $x^*\in Y_+$, $G(x^*)=\emptyset$ otherwise. 
From \eqref{budget set}, \eqref{indirect utility function}, \eqref{demand_map}, and the conventions made, one deduces that $G(x^*)=B(x^*)$, $\mu(x^*)=-v(x^*)$, and $M(x^*)=D(x^*)$  for all $x^*\in X^*$. 

Let $\bar p\in \int Y_+$ and $\bar x \in D(\bar p)\setminus \{0\}$ satisfy the assumptions of the theorem. Using the definitions of Fr\'echet subdifferential and Fr\'echet upperdifferential, one can show that 
\begin{equation*}
\widehat \partial^+\varphi(x^*, x)=\{0\}\times \big(-\widehat \partial u(x)\big) \quad\ ((x^*, x)\in X^*\times X).
\end{equation*}
Hence, the assumption $\widehat \partial u(\bar x)\neq \emptyset$ implies $\widehat \partial^+\varphi(\bar p, \bar x) \neq \emptyset$. Applying Theorem \ref{Frechet upper subdif. MNY09} for $\mu(\cdot)=(-v)(.)$, $M(\cdot)=D(\cdot)$, and $(\bar x, \bar y):=(\bar p, \bar x)$, we have
\begin{equation}\label{ghi}
\widehat \partial(-v)(\bar p)\subset\bigcap_{(x, x^*)\in \widehat \partial^+\varphi(\bar p, \bar x)}\big[x+\widehat D^*B(\bar p, \bar x)(x^*)\big]\\= \bigcap_{x^*\in -\widehat \partial(u)(\bar x)}\widehat D^*B(\bar p, \bar x)(x^*).
\end{equation}
\hskip0.7cm (i) If $\langle \bar p, \bar x\rangle = 1$, then \eqref{ghi} and \eqref{code. of budget map formula} imply \eqref{Fre. subgradients-upper inclusion-a}. 

(ii) If $\langle \bar p, \bar x\rangle < 1$, then by \eqref{code. of budget map formula} one has $\widehat D^*B(\bar p, \bar x)(x^*)\subset \{0\}$ for every $x^*\in X^*$. It follows that
$$\bigcap_{x^*\in -\widehat \partial(u)(\bar x)}\widehat D^*B(\bar p, \bar x)(x^*)\subset \{0\}.$$
 Hence, \eqref{ghi} yields \eqref{Fre. subgradients-upper inclusion-b}.

(iii) If $\langle \bar p, \bar x\rangle <1$ and $\widehat \partial u(\bar x)\setminus N(\bar x, X_+) \neq \emptyset$, then there exist $\bar x^* \in -\widehat \partial u(\bar x)$ such that $\bar x^* \notin -N(\bar x, X_+)$. By \eqref{ghi} and \eqref{code. of budget map formula}, we have 
\begin{equation*}
\widehat \partial(-v)(\bar p)\subset\bigcap_{x^*\in -\widehat \partial(u)(\bar x)}\widehat D^*B(\bar p, \bar x)(x^*)\subset \widehat D^*B(\bar p, \bar x)(\bar x^*)= \emptyset.
\end{equation*}
Hence, \eqref{Fre. subgradients-upper inclusion-c} is valid.

(iv) Now, suppose that $u$ is Fr\'echet differentiable at $\bar x$, and $D:\mbox{dom}B \rightrightarrows X_+$ admits a local upper Lipschitzian selection at $(\bar p, \bar x)$. Since $u$ is Fr\'echet differentiable at $\bar x$, $\nabla \varphi(\bar p, \bar x)=(0, -\nabla u(\bar x))$. By Theorem \ref{Frechet subdif. MNY09}, one gets
\begin{equation}\label{tuv}
\widehat \partial (-v)(\bar p)= \widehat D^*B(\bar p, \bar x)(-\nabla u(\bar x)).
\end{equation}
Since $\widehat D^*B(\bar p, \bar x)(-\nabla u(\bar x))$ can be computed by \eqref{code. of budget map formula} with $x^*:=-\nabla u(\bar x)$ and by Lemma~\ref{neccessary condition lemma}, formula \eqref{Fre. subgradients-upper equality} follows from \eqref{tuv}.
\end{proof}

The upper estimates for the subdifferential $\widehat\partial(-v)(\bar p)$ provided by Theorem \ref{Frechet subdif. HYY17} are sharp. Moreover, under some mild conditions, the set on the right-hand side of \eqref{Fre. subgradients-upper inclusion-a} is a singleton. In addition, if  the indirect utility function $v$ is Fr\'echet differentiable at $\bar p$, then its derivative at $\bar p$ can be easily computed by using \eqref{Fre. subgradients-upper inclusion-a} and \eqref{Fre. subgradients-upper inclusion-b}. The following statement justifies our observations.
 
\begin{corollary} 
Let $\bar p\in \int Y_+$ and $\bar x \in D(\bar p) \setminus \{0\}$ be such that $X_+$ is SNC at $\bar x$, and $u$ is Fr\'echet differentiable at $\bar x$. Then \begin{equation}\label{Fre. subgradients-upper inclusion}
\widehat\partial(-v)(\bar p)\subset \begin{cases}
\{\langle \nabla u(\bar x), \bar x\rangle \bar x\} &\mbox{if } \langle \bar p, \bar x\rangle =1\\
\{0\} & \mbox{if } \langle \bar p, \bar x\rangle < 1.
\end{cases}
\end{equation}
Consequently, if the indirect utility function $v$ is Fr\'echet differentiable at $\bar p$, then
\begin{equation}\label{Fre. equality}
\nabla v(\bar p)=\begin{cases}
-\langle \nabla u(\bar x), \bar x\rangle \bar x &\mbox{if } \langle \bar p, \bar x\rangle =1\\
0 & \mbox{if } \langle \bar p, \bar x\rangle < 1.
\end{cases}
\end{equation}
\end{corollary}
\begin{proof}
Let $\bar p$, $\bar x$, and $u$ satisfy our assumptions. Then one has  $\widehat \partial u(\bar x)=\{\nabla u(\bar x)\}$. Hence, \eqref{Fre. subgradients-upper inclusion} follows from Theorem \ref{Frechet subdif. HYY17} and Lemma~\ref{neccessary condition lemma}.

If $v$ is Fr\'echet differentiable at $\bar p$, then $\widehat\partial v(\bar p)=-\widehat\partial(-v)(\bar p)=\{\nabla v(\bar p)\}$ by \cite[Proposition~1.87]{Mordukhovich_2006a}. So, \eqref{Fre. equality} follows from \eqref{Fre. subgradients-upper inclusion}.

The proof is complete.
\end{proof}
To obtain another corollary from Theorem \ref{Frechet subdif. HYY17}, we now recall a well-known concept in mathematical economics. One says (see, e.g., \cite[p.~1076]{Penot_2014}) that the consumer problem~\eqref{initial problem} satisfies the \textit{non satiety condition} (NSC) if 
 $$\left[(x, p)\in X_+\times Y_+,\; \langle p, x \rangle <1\right]\, \Longrightarrow\, \left[\exists x'\in X_+,\; u(x')>u(x)\right].$$ As it has been noted in \cite{Penot_2014}, NSC is equivalent to the following condition:
$$\left[(x, p)\in X_+\times Y_+,\; \langle p, x \rangle <1\right]\, \Longrightarrow\, u(x)< v(p).$$ Moreover, one can easily prove next lemma, which characterizes NSC via the demand map.

\begin{lemma}\label{NSC} {\rm (See \cite[Lemma~4.2]{Penot_2014})}
The NSC is satisfied if and only if for all $p\in Y_+$, one has $D(p)\subset \{x \in X \,:\, \langle p, x \rangle = 1\}.$
\end{lemma}

Consumer problems with the Cobb-Douglas utility functions satisfy the non satiety condition.
\begin{ex}{\rm 
Suppose that there are $n$ types of available
	goods. The quantities of goods purchased by the
	consumer form the good bundle $x=(x_1, \dots , x_n)$, where $x_i$ is the purchased quantity of the $i$-th good, $i=1, \dots, n$. Assume that each good is perfectly divisible so that any nonnegative quantity can be purchased. Good bundles are vectors in the commodity space $X:=~\R^n$. The set of all possible good bundles
	$$
	X_+:=\big\{x=(x_1, \dots , x_n)\in \R^n \;:\; x_1\geq 0, \dots, x_n \geq 0 \big\}
	$$ 
	is the nonnegative orthant of $\R^n$. The set of prices is 
	$$Y_+=\{p=(p_1, \dots , p_n)\in \R^n \;:\; p_1\geq 0, \dots, p_n \geq 0 \}.$$
	For every $p=(p_1, \dots , p_n)\in Y_+$, $p_i$ is the price of the $i$-th good, $i=1, \dots, n$. 
Given some numbers $A>0$, $\alpha_1, \alpha_2, \dots, \alpha_n \in (0, 1)$, consider the utility function $u:X\to \R_+$ defined by
$u(x)=Ax^{\alpha_1}_1x^{\alpha_2}_2\dots x^{\alpha_n}_n$ for any $x\in \int X_+$. (Recall \cite[p.~96]{Zakon_2017} that $0^\alpha =0$ for any $\alpha>0$.) Clearly, $u$ is strictly increasing in each variable on $\int X_+$. Take $p=(p_1, \dots , p_n)\in Y_+$ and $x=(x_1, \dots , x_n)\in X_+$ satisfying $\langle p, x \rangle <1$. If $x \notin \int X_+$, then we choose $x'=(x'_1, x'_2,\dots, x'_n)$ such that $x'_i \in (0, 1)$ and $\sum_{1}^{n}p_ix'_i \leq 1$. Then $\langle p, x' \rangle \leq 1$ and $x'\in \int X_+$, and therefore $u(x')>0=u(x)$. Consider the case where $x \in \int X_+$. If $p=0$, then by choosing $x'=(x'_1, x_2,\dots, x_n)$ with $x'_1>x_1$, one gets $\langle p, x' \rangle <1$ and $x' \in \int X_+$. Hence, $u(x')>u(x)$ as $u$ is strictly increasing on $\int \R_+$ w.r.t. the first variable. If $p\neq0$, then there exists $i_0$ such that $p_{i_0}>0$. We choose $x'=(x'_1, x'_2,\dots, x'_n)$ with $x'_{i_0}=(1-\sum_{i\neq i_0}p_ix_i)/p_{i_0}$ and $x'_i=x_i$ for all $i\neq i_0$. It follows that $\langle p, x' \rangle =1$, $x' \in \int X_+$, and $x'_{i_0}>x_{i_0}$. As $u$ is strictly increasing on $\int \R_+$ w.r.t. the $i_0$-th variable, one gets $u(x')>u(x)$. We have shown that, for any pair $(p, x)\in Y_+\times X_+$ with $\langle p, x \rangle <1$, there exists $x'\in X_+$ such that $\langle p, x' \rangle \leq 1$ and $u(x')>u(x)$. Thus, the NSC is satisfied. }
\end{ex}

In next corollary, it is not assumed a priori that the demand set $D(\bar p)$ is nonempty.

\begin{corollary}
Suppose that $X_+$ has nonempty interior and the NSC is satisfied. If $u(\cdot)$ is concave and upper semicontinuous on $X_+$, and Fr\'echet differentiable on $D(\bar p)$ then, for any $\bar p \in \int Y_+$, one has
\begin{equation}\label{xyzk}
\widehat\partial (-v)(\bar p)\subset \{\langle \nabla u(\bar x), \bar x\rangle \bar x\;:\; \bar x\in D(\bar p)\}.
\end{equation}
\end{corollary}
\begin{proof}
Let $\bar p \in \int Y_+$. Since $u(\cdot)$ is concave and upper semicontinuous on $X_+$, it is weakly upper semicontinuous on $X_+$. Hence, by \cite[Proposition~4.1]{Penot_2014} we have $D(\bar p)\neq\emptyset$. Take any $\bar x \in D(\bar p)$. Since $X_+$ is a convex set with nonempty interior, it is SNC at $\bar x$. Besides, as the NSC is satisfied, Lemma \ref{NSC} implies $\langle \bar p, \bar x \rangle = 1$. Due to this and the Fr\'echet differentiability of $u(\cdot)$ on $D(\bar p)$, formula \eqref{Fre. subgradients-upper inclusion-a} and Lemma~\ref{neccessary condition lemma} give $\widehat\partial (-v)(\bar p)\subset\{\langle \nabla u(\bar x), \bar x\rangle \bar x\}$. Since $\bar x \in D(\bar p)$ is arbitrarily chosen, from the last inclusion we obtain \eqref{xyzk}. \end{proof}

\subsection{Economic meaning of an exact formula for $\widehat \partial (-v)(\bar p)$}\label{eco meaning 1} 
Suppose that $\bar p \in \int Y_+$ and $\bar\xi \in \widehat \partial (-v)(\bar p)$. By \cite[Theorem~1.88]{Mordukhovich_2006a}, there exists a function $s: X^*\rightarrow \overline \R$ that is finite around $\bar p$, Fr\'echet differentiable at $\bar p$, such that
\begin{equation}\label{variational description}
s(\bar p)=-v(\bar p), \quad \nabla s(\bar p)=\bar\xi, \quad \mbox{and}\quad s(x^*)\le -v(x^*)\quad \mbox{for all}\; x^*\in X^*.
\end{equation}
Fix a vector $q \in X^*$. If $t>0$ is small enough, then the Fr\'echet differentiability of $s$ at $\bar p$ implies
\begin{equation*}
s(\bar p+tq)=s(\bar p)+t\langle \nabla s(\bar p), q\rangle+o(t)
\end{equation*}
with $\disp\lim_{t\rightarrow 0^+}\frac{o(t)}{t}=0$. Combining this with \eqref{variational description} gives
\begin{equation*}
-\langle \bar\xi, q \rangle\geq \dfrac{v(\bar p+tq)-v(\bar p)}{t}+\dfrac{o(t)}{t},
\end{equation*}
for $t>0$ small enough. Hence, we get
\begin{equation*}
-\langle \bar\xi, q \rangle\geq \limsup_{t\rightarrow 0^+}\dfrac{v(\bar p+tq)-v(\bar p)}{t}=: d^+v(\bar p; q),
\end{equation*}where $d^+v(\bar p; q)$ stands for the \textit{upper Dini directional derivative} of $v$ at $\bar p$ in direction $q$. Thus, if $\widehat \partial (-v)(\bar p)$ is nonempty, then
\begin{equation}\label{sharp_estimate}
d^+v(\bar p; q)\leq \inf_{\xi \in \widehat \partial (-v)(\bar p)}[-\langle \xi, q \rangle].
\end{equation}
If a formula for exact computation of $\widehat \partial (-v)(\bar p)$ is available, then \eqref{sharp_estimate} provides us with a sharp upper estimate for the value $d^+v(\bar p; q)$.

Since $\dfrac{v(\bar p+tq)-v(\bar p)}{t}$ is the average rate of the change of the maximal satisfaction, represented by the indirect utility function $v$, of the consumer when the price moves slightly forward direction $q$ from the current price $\bar p$, \textit{the upper Dini directional derivative $d^+v(\bar p; q)$ can be interpreted as an upper bound for the instant rate of the change of the maximal satisfaction of the consumer.} Therefore, the estimate given by \eqref{sharp_estimate} reads as follows: \textit{If the current price is $\bar p$ and the price moves forward a direction $q\in X^*$, then the instant rate of the change of the maximal satisfaction of the consumer is bounded above by the real number $\displaystyle\inf_{\xi \in \widehat \partial (-v)(\bar p)}[-\langle \xi, q\rangle]$.}

If $\widehat \partial (-v)(\bar p)=\emptyset$, then the estimate in \eqref{sharp_estimate} is trivial because $\inf\emptyset=+\infty$. If $u$ is weakly upper semicontinuous and strongly lower semicontinuous on $X_+$, then $v$ is strongly continuous on $\int Y_+$ by \cite[Theorem~3.2]{Huong_Yao_Yen_2016}. In particular, if $X$ is finite-dimensional and $u$ is continuous on  $X_+$, then $v$ is continuous on $\int Y_+$. Another sufficient condition for the continuity of $v$ is the following:  $u$ is concave and strongly continuous on $X_+$ (then $u$ is both weakly upper semicontinuous and  strongly lower semicontinuous on $X_+$). Now, suppose that $v$ is continuous and concave on $\int Y_+$. Then, for every $\bar p\in \int Y_+$, $\widehat \partial (-v)(\bar p)$ and bounded in the weak$^*$ topology by \cite[Prop.~3, p.~199]{Ioffe_Tihomirov_1979}. In fact, since the continuity and concavity of  $v$ on $\int Y_+$ imply that $-v$ is locally Lipschitz and convex on $\int Y_+$ (see, e.g., \cite[Corollary~3.10]{Penot_book2013}), $\widehat \partial (-v)(\bar p)$ is weakly$^*$ compact for every $\bar p\in \int Y_+$.

\subsection{Limiting and singular subgradients}

Next, we will use Theorem 7 from \cite{Mordukhovich_Nam_Yen_2009} to estimate the limiting and singular subdifferentials of the function $-v$. The formulation of that theorem is based on some definitions related to the solution map $M(\cdot)$ of problem \eqref{MNY_problem}. Let $\bar x \in \dom M$ and $\bar y \in M(\bar x)$. One says that $M(\cdot)$ is $\mu$\textit{-inner semicontinuous} at $(\bar x, \bar y)$ if for every sequence $x_k \overset{\mu}{\rightarrow}\bar x$ there is a sequence $y_k \in M(x_k)$ that contains a subsequence converging to $\bar y$. The map $M(\cdot)$ is said to be $\mu$\textit{-inner semicompact} at $\bar x$ if for every sequence $x_k \overset{\mu}{\rightarrow}\bar x$ there is a sequence $y_k \in M(x_k)$ that contains a convergent subsequence.

\begin{theorem}\label{lim. upper subdif. MNY09}{\rm (See \cite[Theorem 7]{Mordukhovich_Nam_Yen_2009})} Let $X, Y$ be Asplund spaces, $M(\cdot): X \rightrightarrows Y$ be the solution map of the parametric problem \eqref{MNY_problem}, and let $(\bar x, \bar y) \in \gph M$ be such that $\varphi$ is lower semicontinuous at $(\bar x, \bar y)$ and $G$ is locally closed around this point. The following statements hold:
\begin{itemize}
\item [{\rm (i)}] Assume that $M(\cdot)$ is $\mu$-inner semicontinuous at $(\bar x, \bar y)$, that either $\epi \varphi$ is SNC at $(\bar x, \bar y,\varphi(\bar x, \bar y))$ or $G$ is SNC at $(\bar x, \bar y)$, and that the qualification condition 
\begin{equation}\label{qualification-Thm 7-MNY09}
\partial^{\infty}\varphi(\bar x, \bar y)\cap (-N((\bar x, \bar y); \gph G))=\{(0, 0)\}
\end{equation}
is satisfied; the above assumptions are automatic if $\varphi$ is locally Lipschitz around $(\bar x, \bar y)$. Then one has the inclusions
\begin{align*}
\partial \mu (\bar x)& \subset \bigcup\left\{x^*+D^*G(\bar x, \bar y)(y^*) : (x^*, y^*) \in \partial \varphi (\bar x, \bar y) \right\},\\
\partial^{\infty} \mu (\bar x)& \subset \bigcup\left\{x^*+D^*G(\bar x, \bar y)(y^*) : (x^*, y^*) \in \partial^{\infty} \varphi (\bar x, \bar y) \right\}.
\end{align*}
\item [{\rm (ii)}] Assume that $M(\cdot)$ is $\mu$-inner semicompact at $\bar x$ and that the other assumption of {\rm (i)} are satisfied at any $(\bar x, \bar y) \in \gph M$. Then one has the inclusions
\begin{align*}
\partial \mu (\bar x)& \subset \bigcup_{\bar y \in M(\bar x)}\left\{x^*+D^*G(\bar x, \bar y)(y^*) : (x^*, y^*) \in \partial \varphi (\bar x, \bar y) \right\},\\
\partial^{\infty} \mu (\bar x)& \subset \bigcup_{\bar y \in M(\bar x)}\left\{x^*+D^*G(\bar x, \bar y)(y^*) : (x^*, y^*) \in \partial^{\infty} \varphi (\bar x, \bar y) \right\}.
\end{align*}
\item [{\rm (iii)}] In addition to {\rm (i)}, assume that $\varphi$ is strictly differentiable at $(\bar x, \bar y)$, the map $M\!:\!\dom G\rightrightarrows Y$ admits a local upper Lipschitzian selection at $(\bar x, \bar y)$, and $G$ is normally regular at $(\bar x, \bar y)$. Then $\mu$ is lower regular at $\bar x$ and 
\begin{equation*}
\partial \mu (\bar x)=\nabla_x\varphi(\bar x, \bar y)+D^*G(\bar x, \bar y)(\nabla_y\varphi(\bar x, \bar y)).
\end{equation*}
\end{itemize}
\end{theorem}

Our results on limiting and singular subdifferentials of $-v$ are stated as follows.

\begin{theorem}\label{lim. subdif. HYY17}
Let $\bar p\in \int Y_+$ and $\bar x \in D(\bar p) \setminus \{0\}$ be such that $D(\bar p)\neq \emptyset$, $X_+$ is SNC at $\bar x$, $u$ is upper semicontinuous at $\bar x$, and $D$ is $v$-inner semicontinuous at $(\bar p, \bar x)$. Assume that either ${\rm hypo}\, u$ is SNC at $(\bar x,\varphi(\bar x))$ or $X$ is finite-dimensional, and the qualification condition
\begin{equation}\label{qualification-2}
\partial^{\infty,+}u(\bar x)\cap N(\bar x, X_+)=\{0\}
\end{equation}
is satisfied.
Then, the following assertions hold:
\begin{itemize}
\item [{\rm (i)}] If $\langle \bar p, \bar x\rangle = 1$, then 
\begin{align}
\partial(-v)(\bar p)&\subset \displaystyle\bigcup_{x^*\in \partial^+ u(\bar x)}\{\lambda \bar x  \;:\; \lambda\geq 0,\; x^*-\lambda \bar p \in N(\bar x, X_+)\}\label{lim. subgradients-upper inclusion-a},\\
\partial^{\infty} (-v) (\bar p) &\subset \bigcup_{x^*\in \partial^{\infty,+}u(\bar x)}\{\lambda \bar x  \;:\; \lambda\geq 0,\; x^*-\lambda \bar p \in N(\bar x, X_+)\}\label{sing. subgradients-upper inclusion-a};
\end{align}
\item[{\rm (ii)}] If $\langle \bar p, \bar x\rangle <1$, then 
\begin{equation}\label{lim. subgradients-upper inclusion-b}
\partial(-v)(\bar p)\subset \{0\},
\end{equation}
\begin{equation}\label{sing. subgradients equality}
\partial^{\infty} (-v) (\bar p)=\{0\};
\end{equation}
\item[{\rm (iii)}] If $\langle \bar p, \bar x\rangle <1$ and $\partial^+ u(\bar x)\cap N(\bar x, X_+) = \emptyset$, then 
\begin{equation}\label{lim. subgradients-upper inclusion-c}
\partial(-v)(\bar p)=\emptyset;
\end{equation}
\item[{\rm (iv)}] If $u$ is strictly differentiable at $\bar x$, and the map $D:\dom B \rightrightarrows X_+$ admits a local upper Lipschitzian selection at $(\bar p, \bar x)$, then $(-v)$ is lower regular at $\bar x$ and
\begin{align}\label{lim. subgradients equality}
\partial (-v)(\bar p)=\begin{cases}
\{\langle \nabla u(\bar x), \bar x\rangle\bar x\}\! &\! \mbox{if}\; \langle \bar p, \bar x\rangle = 1\\
\{0\}&\mbox{if}\; \langle \bar p, \bar x\rangle < 1.
\end{cases}
\end{align}
\end{itemize}
\end{theorem}
\begin{proof}

Let $\bar p\in \int Y_+$ and $\bar x \in D(\bar p)\setminus \{0\}$ satisfy the assumptions of the theorem. At the beginning of the proof of Theorem \ref{Frechet subdif. HYY17}, we have transformed the consumer problem in \eqref{initial problem} to the minimization problem \eqref{MNY_problem}, where  $\varphi(x^*, x)=-u(x)$ for $(x^*, x)\in X^*\times X$, $G(x^*)=B(x^*)$, $\mu(x^*)=-v(x^*)$, and $M(x^*)=D(x^*)$  for all $x^*\in X^*$. 

 Now, we will show that the limiting and singular sudifferential of $-v$ at $\bar p$ can be estimated by assertion (i) of Theorem~\ref{lim. upper subdif. MNY09}. Since $X$ is a reflexive Banach space, so is $X^*$. Hence, $X$ and $X^*$ are Asplund spaces. Moreover, $B$ is locally closed around $(\bar p, \bar x)$ because $\gph B$ is closed and $(\bar p, \bar x)\in \gph B$. Since $u$ is upper semicontinuous at $\bar x$,  $\varphi$ is lower semicontinuous at $(\bar p, \bar x)$. 

If ${\rm hypo}\, u$ is SNC at $(\bar x,\varphi(\bar x))$, then $\epi\varphi$ is SNC at $(\bar p, \bar x, \varphi(\bar p, \bar x))$. If $X$ is finite-dimensional, $B$ is SNC at $(\bar p, \bar x)$. 

Letting $(\bar p, \bar x)$ play the role of $(\bar x, \bar y)$, we now show that \eqref{qualification-2} implies \eqref{qualification-Thm 7-MNY09}. The latter means that
\begin{equation}\label{qualification-Thm 7-MNY09-a}
\partial^{\infty}\varphi(\bar p, \bar x)\cap (-N((\bar p, \bar x); \gph B))=\{(0, 0)\}.
\end{equation}
The inclusion ``$\supset$'' is trivial. Take any $(x, x^*)\in X\times X^*$ belonging to the left-hand side of \eqref{qualification-Thm 7-MNY09-a}. On one hand, since $\partial^{\infty}\varphi(\bar p, \bar x)=\{0\}\times \partial^{\infty}(-u)(\bar x)$, one has $x=0$ and $x^*\in \partial^{\infty}(-u)(\bar x)$. Hence, $-x^*\in \partial^{\infty, +}(-u)(\bar x)$. On the other hand, as $(-x,-x^*)\in N((\bar p, \bar x); \gph B)$ and $x=0$, one has $(0, -x^*)\in N((\bar p, \bar x); \gph B)$. Hence, $0 \in D^*B(\bar p, \bar x)(x^*)$. Combining this with \eqref{code. of budget map formula} implies $-x^*\in N(\bar x, X_+)$. Thus, one has $-x^*\in\partial^{\infty, +}(-u)(\bar x) \cap N(\bar x, X_+)$. So, by \eqref{qualification-2} one obtains $x^*=0$. We have  shown that $(x, x^*)= {(0, 0)}$; hence the inclusion ``$\subset$'' in \eqref{qualification-Thm 7-MNY09-a} is true.

Since all the assumptions for the validity of assertion (i) of Theorem \ref{lim. upper subdif. MNY09} are satisfied, we have
\begin{align}
\partial (-v) (\bar p)& \subset \bigcup\left\{x+D^*B(\bar p, \bar x)(x^*) : (x, x^*) \in \partial \varphi (\bar p, \bar x) \right\},\label{aaa}\\
\partial^{\infty} (-v) (\bar p)& \subset \bigcup\left\{x+D^*B(\bar p, \bar x)(x^*) : (x, x^*) \in \partial^{\infty} \varphi (\bar p, \bar x) \right\}\label{bbb}.
\end{align}
Since $\partial \varphi (\bar p, \bar x)=\{0\}\times \partial(-u)(\bar x)$, and $\partial^{\infty} \varphi (\bar p, \bar x)=\{0\}\times \partial^{\infty}(-u)(\bar x)$, the inclusions \eqref{aaa} and \eqref{bbb} respectively imply 
\begin{align}
\partial (-v) (\bar p)& \subset \bigcup_{x^*\in \partial(-u)(\bar x)}D^*B(\bar p, \bar x)(x^*),\label{ccc}\\
\partial^{\infty} (-v) (\bar p)& \subset \bigcup_{x^*\in \partial^{\infty}(-u)(\bar x)} D^*B(\bar p, \bar x)(x^*)\label{ddd}.
\end{align}
\hskip0.7cm (i) If $\langle \bar p, \bar x\rangle = 1$, then \eqref{ccc}, \eqref{ddd}, and \eqref{code. of budget map formula} imply 
\begin{align*}
\partial(-v)(\bar p)&\subset \displaystyle\bigcup_{x^*\in \partial(-u)(\bar x)}\{\lambda \bar x  \;:\; \lambda\geq 0,\; x^*+\lambda \bar p \in -N(\bar x, X_+)\},\\
\partial^{\infty} (-v) (\bar p)& \subset \bigcup_{x^*\in \partial^{\infty}(-u)(\bar x)}\{\lambda \bar x  \;:\; \lambda\geq 0,\; x^*+\lambda \bar p \in -N(\bar x, X_+)\}.
\end{align*}
As $-\partial(-u)(\bar x)=\partial^+u(\bar x)$ and $-\partial^{\infty}(-u)(\bar x)=\partial^{\infty,+}(-u)(\bar x)$, these inclusions yield \eqref{lim. subgradients-upper inclusion-a} and \eqref{sing. subgradients-upper inclusion-a}, respectively. 

(ii) If $\langle \bar p, \bar x\rangle < 1$, then by \eqref{code. of budget map formula} one has $D^*B(\bar p, \bar x)(x^*)\subset \{0\}$ for every $x^*\in X^*$. It follows that
$$\bigcup_{x^*\in \partial(-u)(\bar x)}D^*B(\bar p, \bar x)(x^*)\subset \{0\}, \quad \bigcup_{x^*\in \partial^{\infty}(-u)(\bar x)} D^*B(\bar p, \bar x)(x^*)\subset \{0\}.
$$
So, \eqref{ccc} implies \eqref{lim. subgradients-upper inclusion-b}, and \eqref{ddd} yields $\partial^{\infty} (-v) (\bar p)\subset \{0\}$.
Remembering that $\partial^{\infty} (-v) (\bar p)$ always contains the origin, one obtains \eqref{sing. subgradients equality}.
 
(iii) If $\langle \bar p, \bar x\rangle <1$ and $\partial^+ u(\bar x)\cap N(\bar x, X_+) = \emptyset$, then for any $x^*\in \partial(-u)(\bar x)$ one has $x^* \notin -N(\bar x, X_+)$. Therefore, by  \eqref{code. of budget map formula}, $D^*B(\bar p, \bar x)(x^*)=\emptyset$ for all $x^*\in \partial(-u)(\bar x)$. Combining this with \eqref{ccc} implies \eqref{lim. subgradients-upper inclusion-c}.

(iv) Since $\bar p \in \int Y_+$ and $\bar x \in B(\bar p)\setminus \{0\}$, $B$ is graphically regular at $(\bar p, \bar x)$ by Theorem \ref{code. of budget map}. Besides, as $u$ is strictly differentiable at $\bar x$, so is $\varphi$ at $(\bar p, \bar x)$ and one has $\nabla \varphi(\bar p, \bar x)=(0, -\nabla u(\bar x))$. Applying assertion (iii) of Theorem \ref{lim. upper subdif. MNY09}, we have
\begin{equation}\label{kkk}
\partial (-v)(\bar p)= D^*B(\bar p, \bar x)(-\nabla u(\bar x)).
\end{equation}
Since $D^*B(\bar p, \bar x)(-\nabla u(\bar x))$ can be computed via \eqref{code. of budget map formula} with $x^*:=-\nabla u(\bar x)$ and  Lemma~\ref{neccessary condition lemma}, formula \eqref{lim. subgradients equality} follows from \eqref{kkk}.

The proof is complete. \end{proof}
We now give an illustrative example for Theorem \ref{lim. subdif. HYY17}.
\begin{ex}{\rm 
Choose $X=\R, X_+=\R_+$, and define the concave utility function $u: X\rightarrow \R$ by
\begin{equation*}
 u(x)=\begin{cases}
 x & \mbox{if }\quad x \leq 1\\
 1 & \mbox{if }\quad x>1.
 \end{cases}
 \end{equation*}
One has $Y_+=\R_+$ and
 \begin{equation*}
 B(p)=\begin{cases}
 [0, +\infty) & \mbox{if }\quad p=0\\
 [0, 1/p]& \mbox{if }\quad p>0.
 \end{cases}
 \end{equation*}
It is easy to show that
\begin{equation*}
v(p)=\begin{cases}
1 & \mbox{if } \quad 0 \leq p <1\\
1/p & \mbox{if } \quad p\geq 1
\end{cases}
\end{equation*}
and
 \begin{equation*}
D(p)=\begin{cases}
[1, +\infty ) &\mbox{if }\quad  p=0\\
[1, 1/p] &\mbox{if }\quad 0<p<1 \\
\{1/p\}  &\mbox{if } \quad p\geq 1.
\end{cases}
\end{equation*}
\begin{figure}[htbp]	
	\includegraphics[scale=.25]{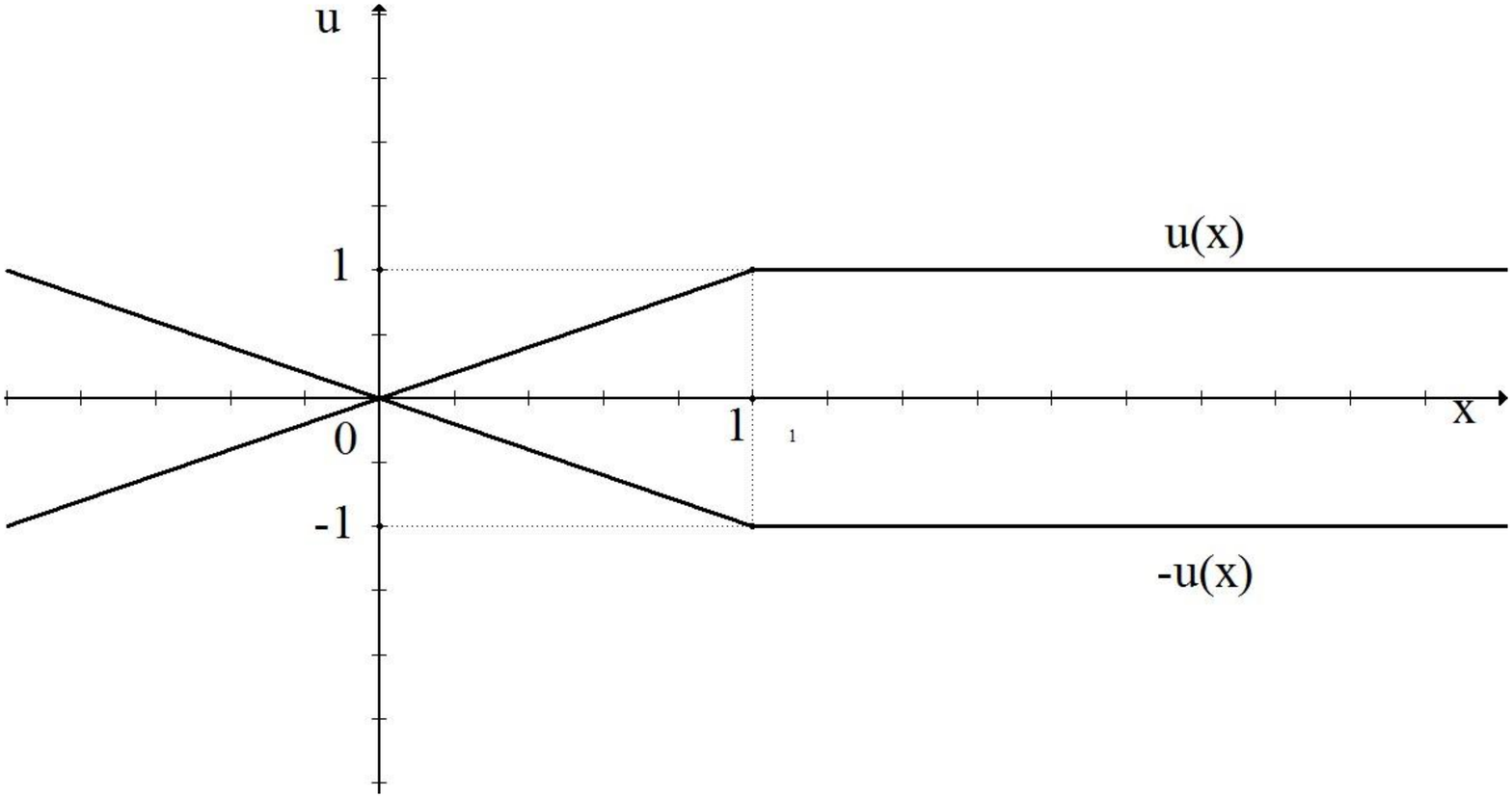}
	\includegraphics[scale=.25]{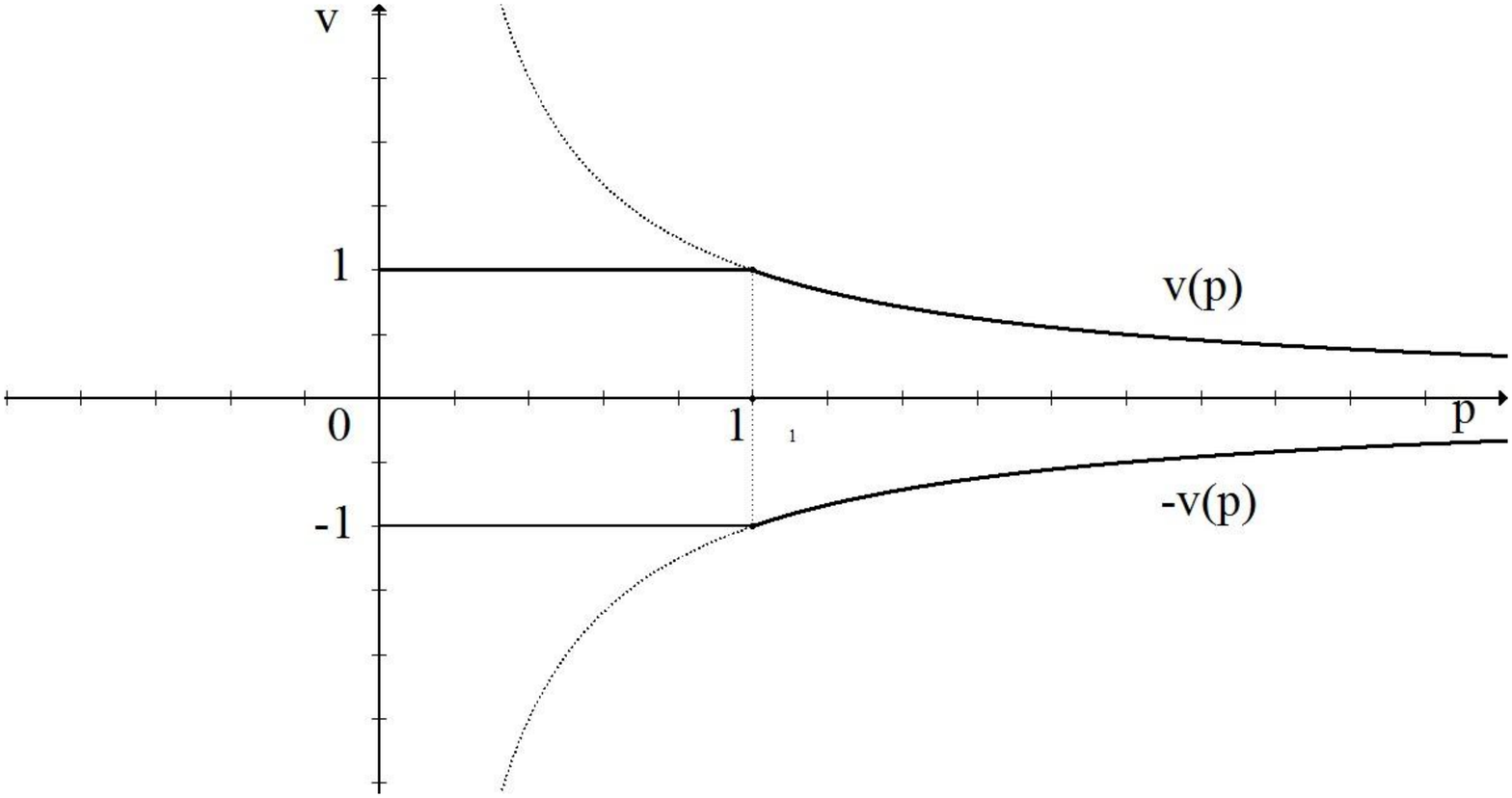}
	\centerline{Figure 1: Utility function and indirect utility function}
	\end{figure} 
	\begin{figure}[htbp]	
		\includegraphics[scale=.25]{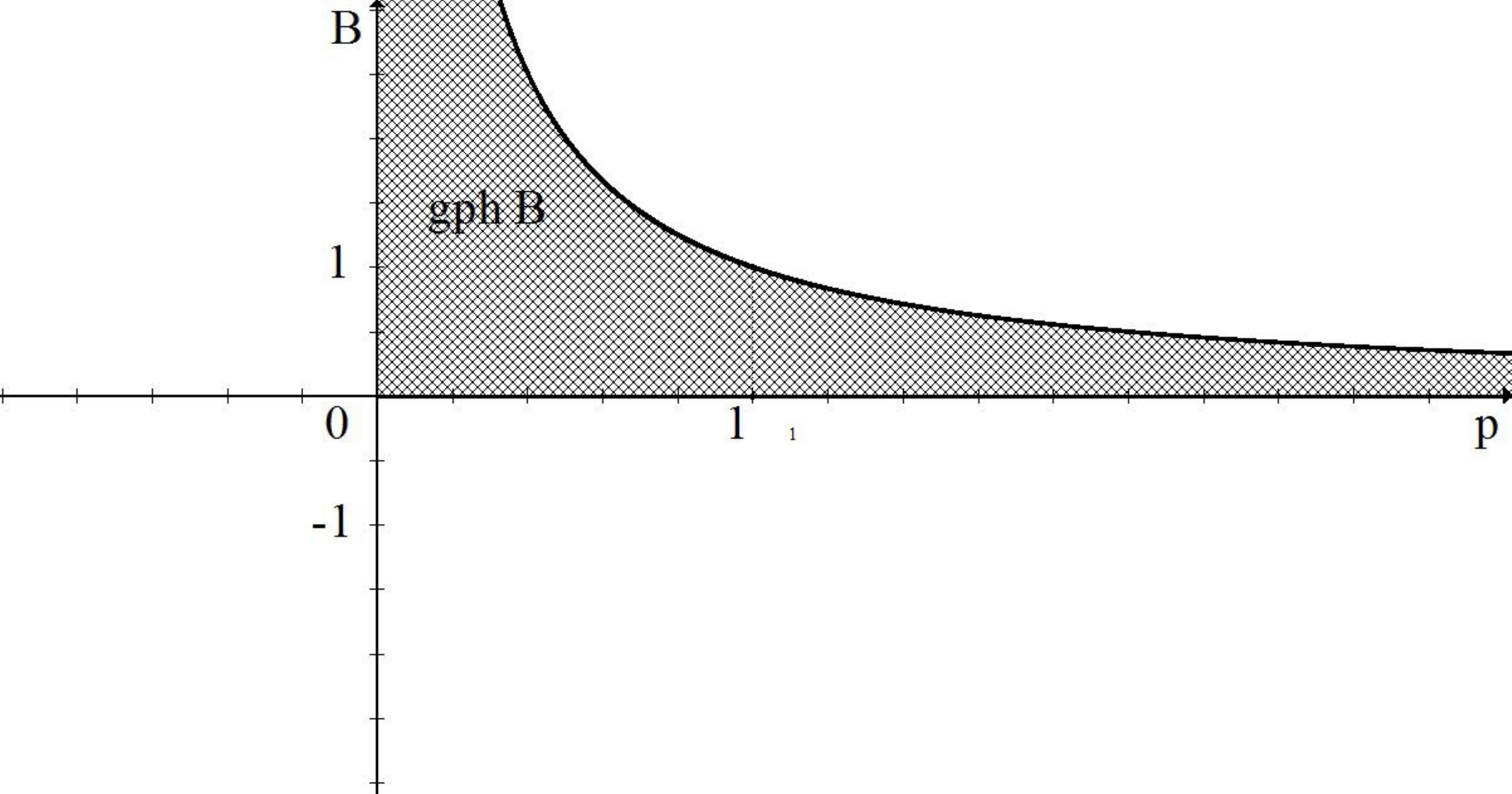}
		\includegraphics[scale=.25]{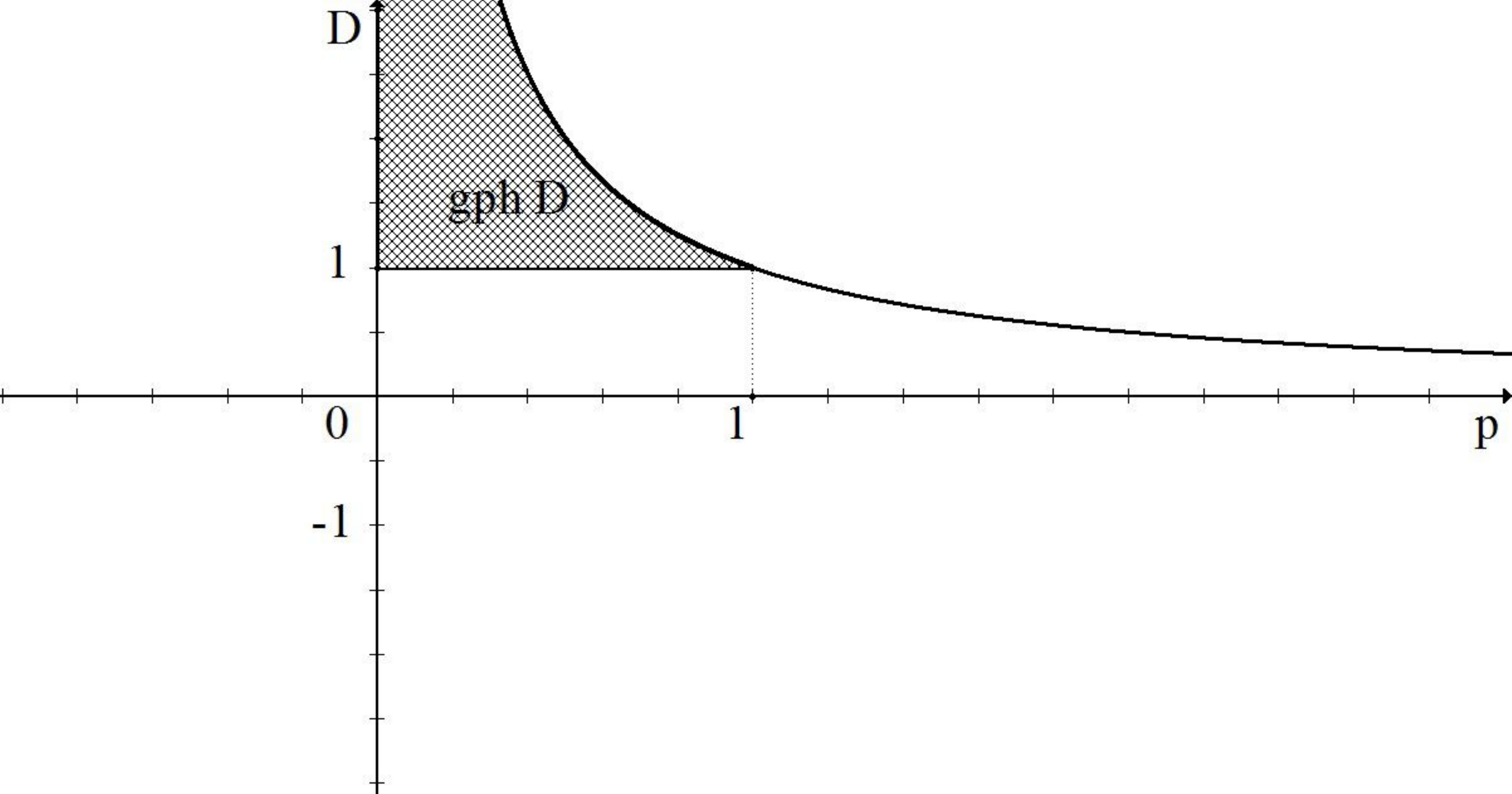}
		\centerline{Figure 2: Budget map and demand map}
	\end{figure} 
	
We see that both functions $v$ and $-v$ are neither convex nor concave on $X_+$. Consider the pair $(\bar p, \bar x)=(1, 1) $ belonging to  $\gph D$. One has $\bar p\in \int Y_+, \bar x \neq 0$, $X_+$ is SNC at $\bar x$, and $\langle \bar p, \bar x\rangle =1$. Besides, as $u$ is locally Lipschitz on $X$, $u$ is upper semicontinuous at $\bar x$ and $\partial^{\infty, +}u(\bar x)=\{0\}$. Thus, the qualification \eqref{qualification-2} is satisfied. Moreover, the formulas of $v$ and $D$ imply that $D$ is $v$-inner semicontinuous at $(\bar p, \bar x)$. 

It follows from definitions that
\begin{equation*}
\partial^+u(\bar x)=[0, 1], \quad \partial (-v)(\bar p)=[0, 1],\quad \mbox{and} \quad \partial^\infty (-v)(\bar p)=\{0\} 
\end{equation*}
while, by direct computation, we have
\begin{align*}
& \displaystyle\bigcup_{x^*\in \partial^+ u(\bar x)}\{\lambda \bar x  \;:\; \lambda\geq 0,\; x^*-\lambda \bar p \in N(\bar x, X_+)\}= [0, 1],\\
& \displaystyle\bigcup_{x^*\in \partial^{\infty,+} u(\bar x)}\{\lambda \bar x  \;:\; \lambda\geq 0,\; x^*-\lambda \bar p \in N(\bar x, X_+)\}= \{0\}.
\end{align*}
Thus, the inclusions \eqref{lim. subgradients-upper inclusion-a} and \eqref{sing. subgradients-upper inclusion-a}  hold as equalities.
}
\end{ex}

Finally, we present a counterpart of Theorem \ref{lim. subdif. HYY17}, where the assumption on the $v$-inner semicontinuity of $D$ at $(\bar p,\bar x)$ is removed. In fact, here one has the $v$-inner semicompactness of $D$ at $\bar p$, which is guaranteed by the assumptions saying that $X$ is finite-dimensional and $\bar p\in \int Y_+$.

\begin{theorem}\label{lim. subdif. HYY17-finite-dim.}
Suppose that $X$ is a finite-dimensional Banach space, the non satiety condition is satisfied, and $u$ is upper semicontinuous on $X_+$. For any $\bar p \in \int Y_+$, if the qualification condition \eqref{qualification-2} is satisfied at every $\bar x \in D(\bar p)$, then  one has
\begin{align}
\partial(-v)(\bar p)&\subset \displaystyle\bigcup_{\bar x \in D(\bar p)}\,\bigcup_{x^*\in \partial^+ u(\bar x)}\{\lambda \bar x  \;:\; \lambda\geq 0,\; x^*-\lambda \bar p \in N(\bar x, X_+)\}\label{lim. subgradients-upper inclusion-d},\\
\partial^{\infty} (-v) (\bar p) &\subset \bigcup_{\bar x \in D(\bar p)}\, \bigcup_{x^*\in \partial^{\infty,+}u(\bar x)}\{\lambda \bar x  \;:\; \lambda\geq 0,\; x^*-\lambda \bar p \in N(\bar x, X_+)\}\label{sing. subgradients-upper inclusion-d}.
\end{align}
\end{theorem}
\begin{proof}
Fix a vector $\bar p\in \int Y_+$ and suppose that all the assumptions of the theorem are satisfied. We transform \eqref{initial problem} to \eqref{MNY_problem} in same manner as in the proofs of Theorem~\ref{Frechet subdif. HYY17} and Theorem~\ref{lim. subdif. HYY17}.
Let us show that the limiting and singular sudifferential of $-v$ at $\bar p$ can be estimated by assertion (ii) of Theorem~\ref{lim. upper subdif. MNY09}. 
To do so, we have to prove that $D$ is $v$-inner semicompact at $\bar p$. 

On one hand, for a given $r \in (0, 1)$, as shown in the proof of Proposition~3.2 in \cite{Huong_Yao_Yen_2016}, the vector $\bar q:=r\bar p$ belongs to ${\rm int}\,Y_+$, the set $U_1:=\bar q+{\rm int}\,Y_+$ is an open neighborhood of $\bar p$, and $B(p) \subset B(\bar q)$ for every $p \in U_1$. On the other hand, since $X$ is a finite-dimensional Banach space, the weak topology of $X$ coincides with its norm topology. Hence, by the upper semicontinuity of $u$ on $X_+$ and \cite[Proposition~4.1]{Penot_2014}, $D(p)$ is nonempty for every $p \in \int Y_+$. So, the set $U:=U_1\cap \int Y_+$ is an open neighborhood of $\bar p$ and  one has $\emptyset \neq D(p) \subset B(\bar q)$ for every $p \in U$. Take any sequence $p_k \overset{v}{\rightarrow}\bar p$. Without loss of generality, one may assume that $p_k\in U$ for all $k$. Thus, for each $k$, one can select a vector $x_k\in D(p_k)$. Since $\{x_k\} \subset B(\bar q)$, the compactness of $B(\bar q)$ (see \cite[Proof of Proposition~4.1]{Penot_2014}) implies that $\{x_k\}_{k\in \N}$ contains a convergent subsequence. We have thus proved that $D$ is $v$-inner semicompact at $\bar p$.

By assertion (ii) of Theorem \ref{lim. upper subdif. MNY09} where $\varphi(x^*, x)=-u(x)$ for $(x^*, x)\in X^*\times X$, $G(x^*)=B(x^*)$, $\mu(x^*)=-v(x^*)$, $M(x^*)=D(x^*)$  for $x^*\in X^*$, and $\bar x=\bar p$, we have 
\begin{align*}
\partial (-v) (\bar p)& \subset \bigcup_{\bar x \in D(\bar p)}\left\{x+D^*B(\bar p, \bar x)(x^*) : (x, x^*) \in \partial \varphi (\bar p, \bar x) \right\},\\
\partial^{\infty} (-v) (\bar p)& \subset \bigcup_{\bar x \in D(\bar p)}\left\{x+D^*B(\bar p, \bar x)(x^*) : (x, x^*) \in \partial^{\infty} \varphi (\bar p, \bar x) \right\}.
\end{align*}
Since $\partial \varphi (\bar p, \bar x)=\{0\}\times \partial(-u)(\bar x)$, and $\partial^{\infty} \varphi (\bar p, \bar x)=\{0\}\times \partial^{\infty}(-u)(\bar x)$, these inclusions imply 
\begin{align}
\partial (-v) (\bar p)& \subset \bigcup_{\bar x \in D(\bar p)}\, \bigcup_{x^*\in \partial(-u)(\bar x)}D^*B(\bar p, \bar x)(x^*),\label{iii}\\
\partial^{\infty} (-v) (\bar p)& \subset \bigcup_{\bar x \in D(\bar p)}\, \bigcup_{x^*\in \partial^{\infty}(-u)(\bar x)} D^*B(\bar p, \bar x)(x^*)\label{jjj}.
\end{align}
As the NSC is satisfied, we have $\bar x \neq 0$ and $\langle \bar p, \bar x\rangle = 1$  for any $\bar x\in D(\bar p)$. Thus, it follows from \eqref{iii}, \eqref{jjj}, and \eqref{code. of budget map formula} that 
\begin{align*}
\partial (-v) (\bar p)& \subset \bigcup_{\bar x \in D(\bar p)} \bigcup_{x^*\in \partial(-u)(\bar x)}\{\lambda \bar x  \;:\; \lambda\geq 0,\; x^*+\lambda \bar p \in -N(\bar x, X_+)\}\\
\partial^{\infty} (-v) (\bar p)& \subset \bigcup_{\bar x \in D(\bar p)} \bigcup_{x^*\in \partial^{\infty}(-u)(\bar x)} \{\lambda \bar x  \;:\; \lambda\geq 0,\; x^*+\lambda \bar p \in -N(\bar x, X_+)\}.
\end{align*}
Remembering that $-\partial(-u)(\bar x)=\partial^+u(\bar x)$ and $-\partial^{\infty}(-u)(\bar x)=\partial^{\infty,+}(-u)(\bar x)$, one obtains \eqref{lim. subgradients-upper inclusion-d} and \eqref{sing. subgradients-upper inclusion-d} from these inclusions and thus completes the proof.
\end{proof}
\begin{corollary}
Let $\bar p \in \int Y_+$. If $X$ is finite-dimensional, the NSC is satisfied, and $u$ is strictly differentiable on $D(\bar p)$, then
\begin{equation}
\partial(-v)(\bar p)\subset \{\langle\nabla u(\bar x), \bar x\rangle \bar x  \;:\; \bar x \in D(\bar p)\},\label{lim. subgradients-upper inclusion-NSC}
\end{equation}
\begin{equation}
\partial^{\infty} (-v) (\bar p) = \{0\}.\label{sing. subgradients-upper inclusion-NSC}
\end{equation}
\end{corollary}
\begin{proof}
Under the assumptions of the corollary, take any $\bar x \in D(\bar p)$. By the strict differentiability of $u$ on $D(\bar p)$, one has $\partial^+u(\bar x)=\{\nabla u(\bar x)\}$ (see \cite[Corollary 1.82]{Mordukhovich_2006a}), and $u$ is locally Lipschitz around $\bar x$ (see \cite[p. 19]{Mordukhovich_2006a}). Therefore, $\partial^{\infty,+}u(\bar x)=\{0\}$; hence the qualification condition \eqref{qualification-2} is satisfied at $\bar x$. It follows from Theorem~\ref{lim. subdif. HYY17-finite-dim.} that 
\begin{align}
\partial(-v)(\bar p)&\subset \displaystyle\bigcup_{\bar x \in D(\bar p)}\{\lambda \bar x  \;:\; \lambda\geq 0,\; \lambda \bar p \in  \nabla u(\bar x)- N(\bar x, X_+)\}\label{mmm},\\
\partial^{\infty} (-v) (\bar p) &\subset \bigcup_{\bar x \in D(\bar p)} \{\lambda \bar x  \;:\; \lambda\geq 0,\; \lambda \bar p \in -N(\bar x, X_+)\}\label{nnn}.
\end{align}
For any $\bar x \in D(\bar p)$, by Lemma~\ref{neccessary condition lemma} one has $$\{\lambda\geq 0\;:\; \lambda \bar p \in  \nabla u(\bar x)- N(\bar x, X_+)\} = \{\langle\nabla u(\bar x), \bar x\rangle\},$$ while $\{\lambda\geq 0\;:\; \lambda \bar p \in -N(\bar x, X_+)\}=\{0\}$. Thus, \eqref{lim. subgradients-upper inclusion-NSC} follows from \eqref{mmm}, and \eqref{sing. subgradients-upper inclusion-NSC} follows from \eqref{nnn}, because $0 \in \partial^{\infty} (-v) (\bar p)$.

The proof is complete.
\end{proof}

\subsection{Economic meaning of Theorems \ref{lim. subdif. HYY17} and \ref{lim. subdif. HYY17-finite-dim.}}

Let $\Omega\subset X$ be a nonempty subset in a Banach space. One defines the \textit{distance function} $d_\Omega(\cdot): X\to \R$ by setting $d_\Omega(x)=\inf\left\{\lVert x-u\rVert \,:\ u\in \Omega \right\}$. If $\Omega$ is closed, then $x\in\Omega$ if and only if $d_\Omega(x)=0$. Moreover, $d_\Omega(.)$ is Lipschitz on $X$ (see \cite[Proposition~2.4.1]{Clarke_1990}) with the Lipschitz constant 1, i.e., $|d_\Omega(x')-d_\Omega(x)|\leq \| x'-x\|$ for any $x, x' \in X$. The \textit{Clarke tangent cone} (see \cite[p.~51]{Clarke_1990}) to $\Omega$ at  $\bar x\in \Omega$  is the set \begin{align*}
T_C(\bar x; \Omega):=\left\{v\in X \;:\; d^0_\Omega(\bar x;v)=0\right\},
\end{align*}where 
\begin{equation*}
d^0_\Omega(\bar x;v):=\limsup_{\stackrel{\,t\to 0^+}{x\to\bar x}}\dfrac{d_\Omega(x+tv)-d_\Omega(x)}{t}
\end{equation*}
is the generalized  Clarke derivative of $d_\Omega(\cdot)$ at $\bar x$ in direction $v$. The \textit{Clarke normal cone} to $\Omega$ at $\bar x$ is defined by \begin{align*}
N_C(\bar x; \Omega):=\big\{x^*\in X^* \;:\; \langle x^*, v \rangle \le 0,\ \forall v \in T_C(\bar x; \Omega)\}.
\end{align*}

 Let $\varphi: X \rightarrow \overline\R$ be finite at $\bar x \in X$. The \textit{Clarke subdifferential} of $\varphi$ at $\bar x$ is given by
 \begin{equation*}
\partial_C \varphi(\bar x):=\left\{x^*\in X \;:\; (x^*, -1)\in N_C\big((\bar x, \varphi(\bar x)); \epi \varphi \big)\right\}.
 \end{equation*}
 The relationships between the Clarke subdifferential and the Mordukhovich subdifferential is as follows.
 \begin{theorem}{\rm (See \cite[Theorem~3.57]{Mordukhovich_2006a})}
 Suppose that $X$ is an Asplund space and $\varphi: X \rightarrow \overline\R\setminus\{-\infty\}$ is finite and lower semicontinuous at $\bar x \in X$. Then
 \begin{equation}\label{Clar via Mor}
 \partial_C \varphi(\bar x)=\cl^*[\co \partial\varphi(\bar x)+\co\partial^\infty \varphi(\bar x)]=\cl^*\co[\partial\varphi(\bar x)+\partial^\infty \varphi(\bar x)],
 \end{equation}
 where $\co A$ stands for the convex closure of $A\subset X^*$ and $\cl^*A$ denotes the weak$^*$ topological closure of $A$.
 \end{theorem}
 
Following \cite{Rockafellar_1980}, we write
$
(x, \alpha)\downarrow_\varphi \bar x
$ if $(x, \alpha) \in \epi \varphi$, $x \to \bar x$, and $\alpha \to \varphi(\bar x)$. The \textit{upper subderivative} of $\varphi$ at $\bar x$ with respect to $v\in X$ (see \cite[formula~(4.2)]{Rockafellar_1980}, or \cite[formula~(2.3)]{Rockafellar_1979}) is defined by  
\begin{equation*}\label{upper subderivative}
\varphi^\uparrow(\bar x; v):=\lim_{\varepsilon\to 0^+}\limsup_{\stackrel{t\to 0^+}{(x, \alpha)\downarrow_\varphi \bar x}}\,\inf_{w\in v+\varepsilon B}\dfrac{\varphi(x+tw)-\alpha}{t}.
\end{equation*}
By \cite[Theorem~2]{Rockafellar_1980}, the function $v\to \varphi^\uparrow(\bar x; v)$ is convex, homogeneous, lower semicontinuous, not identical $+\infty$. When $\varphi$ is lower semicontinuous at $\bar x$, $\varphi^\uparrow(\bar x; v)$ can be given by a slightly simpler expression
\begin{equation*}
\varphi^\uparrow(\bar x; v):=\lim_{\varepsilon\to 0^+}\limsup_{\stackrel{t\to 0^+}{x\overset{\varphi}{\to} \bar x}}\,\inf_{w\in v+\varepsilon B}\dfrac{\varphi(x+tw)-\varphi(x)}{t}.
\end{equation*}
\begin{theorem}{\rm (See \cite[p.~97]{Clarke_1990})}
Let $\varphi: X \rightarrow \overline\R$ be finite at $\bar x$. One has $\partial_C \varphi(\bar x)=\emptyset$ if and only if $\varphi^\uparrow(\bar x; 0)=-\infty$. Otherwise, one has 
\begin{equation*}
\partial_C \varphi(\bar x)=\{x^*\in X^* \;:\; \varphi^\uparrow(\bar x; v)\geq \langle x^*, v\rangle,\ \forall v\in X\},
\end{equation*}
and 
\begin{equation}\label{subderivative estimate}
\varphi^\uparrow(\bar x; v)=\sup\{\langle x^*, v\rangle \;:\; x^*\in \partial_C \varphi(\bar x)\}
\end{equation}
for any $v\in X$.
\end{theorem}
The \textit{generalized Clarke derivative} of $\varphi$ at $\bar x$ with respect to  $v\in X$ (see \cite[formula~(5.2)]{Rockafellar_1980}) is 
\begin{equation*}
\varphi^0(\bar x; v):=\limsup_{\stackrel{ t\to 0^+}{(x, \alpha)\downarrow_\varphi \bar x}}\dfrac{\varphi(x+tv)-\alpha}{t}.
\end{equation*}
If $\varphi$ is lower semicontinuous at $\bar x$, then the above formula becomes
\begin{equation*}
\varphi^0(\bar x; v):=\limsup_{\stackrel{t\to 0^+}{x \overset{\varphi}{\to} \bar x}}\dfrac{\varphi(x+tv)-\varphi(x)}{t}.
\end{equation*}
The function $\varphi$ is said to be \textit{directionally Lipschitzian} at $\bar x$ with respect to $v\in X$ (see \cite[formula~(5.2)]{Rockafellar_1980}) if $f$ is finite at $x$ and 
\begin{equation*}
\limsup_{\stackrel{w\to v,\; t\to 0^+}{(x, \alpha)\downarrow_\varphi \bar x}}\dfrac{\varphi(x+tw)-\alpha}{t}<+\infty,
\end{equation*}
a condition which can be simplified when $f$ is lower semicontinuous at $x$ to 
\begin{equation*}
\limsup_{\stackrel{w\to v,\; t\to 0^+}{x \overset{\varphi}{\to} \bar x}}\dfrac{\varphi(x+tw)-\varphi(x)}{t}<+\infty.
\end{equation*}
It follows from the definition that $f$ is Lipschitz around $\bar x$ if and only if it is directionally Lipschitzian at $\bar x$ w.r.t. $v=0$. One says that $f$ is \textit{directionally Lipschitzian} at $\bar x$ if there is at least one $v$, not necessarily $0$, such that $f$ is directionally Lipschitzian at $\bar x$ w.r.t. $v$. We refer to \cite[Sect.~6]{Rockafellar_1980} for some conditions guaranteeing that $f$ is directionally Lipschitzian at $\bar x$.

\begin{theorem}\label{Rock. Theorem}{\rm (See \cite[Theorem 3]{Rockafellar_1980})}
Let $\varphi: X \rightarrow \overline\R$ be finite at $\bar x$. If $\varphi$ is directionally Lipschitzian at $\bar x$ w.r.t. $v \in X$, then
$
\varphi^\uparrow(\bar x; v)=\varphi^0(\bar x; v).
$
\end{theorem}

Turning back to the indirect utility function $v$ of \eqref{initial problem}, we suppose that $v$ is finite at $\bar p \in \int Y_+$. Combining formulas \eqref{lim. subgradients-upper inclusion-a} and \eqref{sing. subgradients-upper inclusion-a} (resp., \eqref{lim. subgradients-upper inclusion-d} and \eqref{sing. subgradients-upper inclusion-d}), we obtain an upper estimate for the sum $\partial(-v) (\bar p)+\partial^{\infty} (-v) (\bar p)$. Thus, if $v$ is upper semicontinuous at $\bar p$ and $v(p)\neq +\infty$ for all $p\in Y_+$, then by \eqref{Clar via Mor} we get an upper estimate for the Clarke subdifferential $\partial_C(-v)(\bar p)$. In other words, {\textit{we find a subset $A$ of $X$ in an explicit form such that $\partial_C(-v)(\bar p)\subset A$.} Since $0\in\partial^{\infty} (-v) (\bar p)$, one has $\partial_C(-v)(\bar p)\neq\emptyset$ if and only if $\partial(-v) (\bar p)\neq\emptyset$. Therefore, if the latter is valid then, for any  $q\in X^*$, by  \eqref{subderivative estimate} one has
\begin{equation*}\label{Rock. estimate 1}
(-v)^\uparrow(\bar p; q)=\sup_{\xi \in \partial_C(-v)(\bar p)}\langle \xi, q\rangle;
\end{equation*}
hence
\begin{equation*}\label{Rock. estimate 1'}
(-v)^\uparrow(\bar p; q)\leq \displaystyle\sup_{\xi \in A}\,\langle \xi, q\rangle=:\alpha(q).
\end{equation*} 
So, if $v$ is upper semicontinuous at $\bar p$, then
\begin{equation}\label{Rock. estimate 2}
\lim_{\varepsilon\to 0^+}\limsup_{\stackrel{t\to 0^+}{p\overset{v}{\to} \bar p}}\,\inf_{w\in q+\varepsilon B}\dfrac{(-v)(p+tw)-(-v)(p)}{t}=\sup_{\xi \in \partial_C(-v)(\bar p)}\langle \xi, q\rangle,
\end{equation}
and
\begin{equation}\label{Rock. estimate 2'}
\lim_{\varepsilon\to 0^+}\limsup_{\stackrel{t\to 0^+}{p\overset{v}{\to} \bar p}}\,\inf_{w\in q+\varepsilon B}\dfrac{(-v)(p+tw)-(-v)(p)}{t}\leq \alpha(q).
\end{equation}
If, in addition, $-v$ is directionally Lipschitzian at $\bar p$ w.r.t. $q$, then by Theorem \ref{Rock. Theorem} one has $(-v)^\uparrow(\bar p; q)=(-v)^0(\bar p; q)$. Hence, \eqref{Rock. estimate 2} and \eqref{Rock. estimate 2'} respectively yield
\begin{equation}\label{Rock. estimate 3}
\limsup_{\stackrel{t\to 0^+}{p \overset{v}{\to} \bar p}}\dfrac{(-v)(p+tq)-(-v)(p)}{t}=\sup_{\xi \in \partial_C(-v)(\bar p)}\langle \xi, q\rangle,
\end{equation}
and 
\begin{equation}\label{Rock. estimate 3a}
\limsup_{\stackrel{t\to 0^+}{p \overset{v}{\to} \bar p}}\dfrac{(-v)(p+tq)-(-v)(p)}{t}\leq \alpha(q).
\end{equation}
Formula \eqref{Rock. estimate 3} is equivalent to 
\begin{equation*}
\liminf_{\stackrel{t\to 0^+}{p \overset{v}{\to} \bar p} }\dfrac{v(p+tq)-v(p)}{t}=-\sup_{\xi \in \partial_C(-v)(\bar p)}\langle \xi, q\rangle,
\end{equation*}
and therefore
\begin{equation*}
d^-v(\bar p; q):=\liminf_{ t\to 0^+}\dfrac{v(\bar p+tq)-v(\bar p)}{t}\geq-\sup_{\xi \in \partial_C(-v)(\bar p)}\langle \xi, q\rangle,
\end{equation*}
with $d^-v(\bar p; q)$ denoting the \textit{lower Dini directional derivative} of $v$ at $\bar p$ in direction $q$.
Similarly, \eqref{Rock. estimate 3a} is equivalent to
\begin{equation*}
\liminf_{\stackrel{t\to 0^+}{p \overset{v}{\to} \bar p}}\dfrac{v(p+tq)-v(p)}{t}\geq-\alpha (q),
\end{equation*}
and thus we have the following lower estimate for the lower Dini directional derivative of $v$ at $\bar p$ in direction $q$:
\begin{equation}\label{lower Dini estimate}
d^-v(\bar p; q)\geq -\alpha(q).
\end{equation}

In Subsection \ref{eco meaning 1}, we have explained the economic meaning of an upper estimate for the upper Dini directional derivative $d^+v(\bar p; q)$. Analogously, \textit{the lower Dini directional derivative $d^+v(\bar p; q)$ can be interpreted as an lower bound for the instant rate of the change of the maximal satisfaction of the consumer.} Therefore, the estimate given by \eqref{lower Dini estimate} reads as follows: \textit{If the current price is $\bar p$ and the price moves forward a direction $q\in X^*$, then the instant rate of the change of the maximal satisfaction of the consumer is bounded below by the real number $-\alpha(q)=\displaystyle\inf_{\xi \in A}\,\langle -\xi, q\rangle$. Here,  the set $A$ is an upper bound for the Clarke subdifferential $\partial_C(-v)(\bar p)$, which is provided either by Theorem \ref{lim. subdif. HYY17}, or by Theorem \ref{lim. subdif. HYY17-finite-dim.}.}

\end{document}